\documentclass[12pt]{amsart}

\usepackage{amsmath,amsfonts,amssymb,amsthm,mathabx,shuffle,latexsym,mathtools,stmaryrd}
\usepackage{enumitem}
\usepackage[usenames, dvipsnames, svgnames]{xcolor}
\usepackage[backref=page]{hyperref}
\usepackage{ytableau}
\usepackage{tikz-cd}
\usepackage{pbox}

\usepackage{tikz}
\usetikzlibrary{calc, shapes, backgrounds,arrows,positioning,backgrounds}
\tikzset{>=stealth',
  head/.style = {fill = white, text=black},
  plaque/.style = {draw, rectangle, minimum size = 9mm, fill=Gainsboro}, 
   Kplaque/.style = {draw, rectangle, minimum size = 9mm, fill=white}, 
  newplaque/.style = {draw=red, ellipse, minimum size = 9mm,ultra thick, fill=white}, 
  posex/.style={->,thick},
  lift/.style={right hook->},
  beta/.style={dashed, <->},
  pil/.style={->,thick},
  junct/.style = {draw,circle,inner sep=0.5pt,outer sep=0pt, fill=black}
  }

\newtheorem{theorem}{Theorem}[section]
\newtheorem{lemma}[theorem]{Lemma}
\newtheorem{proposition}[theorem]{Proposition}
\newtheorem{corollary}[theorem]{Corollary}
\newtheorem{conjecture}[theorem]{Conjecture}

\theoremstyle{definition}
\newtheorem{definition}[theorem]{Definition}
\newtheorem{examplex}[theorem]{Example}
\newenvironment{example}
  {\pushQED{\qed}\examplex}
  {\popQED\endexamplex}

\theoremstyle{remark}
\newtheorem{remark}[theorem]{Remark}

\numberwithin{equation}{section}

\DeclareMathOperator{\ex}{ex}
\newcommand{\wt}{\ensuremath{\mathrm{wt}}}
\newcommand{\kwt}{\ensuremath{\mathrm{kwt}}}

\newcommand{\destand}{\ensuremath{\mathrm{dst}}}

\newcommand{\x}{\ensuremath{\mathbf{x}}}

\newcommand{\flags}{\ensuremath{\mathsf{Flags}}}

\newcommand{\Sym}{\ensuremath{\mathrm{Sym}}}
\newcommand{\QSym}{\ensuremath{\mathrm{QSym}}}
\newcommand{\Poly}{\ensuremath{\mathrm{Poly}}}

\newcommand{\schub}{\ensuremath{\mathfrak{S}}}
\newcommand{\groth}{\ensuremath{\overline{\mathfrak{S}}}}
\newcommand{\slide}{\ensuremath{\mathfrak{F}}}
\newcommand{\glide}{\ensuremath{\overline{\mathfrak{F}}}}
\newcommand{\atom}{\ensuremath{\mathfrak{A}}}
\newcommand{\lascouxatom}{\ensuremath{\overline{\mathfrak{A}}}}
\newcommand{\kaon}{\ensuremath{\overline{\mathfrak{P}}}}
\newcommand{\particle}{\ensuremath{\mathfrak{P}}}
\newcommand{\fundamental}{\ensuremath{F}}
\newcommand{\multifundamental}{\ensuremath{\overline{F}}}
\newcommand{\key}{\ensuremath{\mathfrak{D}}}
\newcommand{\qkey}{\ensuremath{\mathfrak{Q}}}
\newcommand{\lascoux}{\ensuremath{\overline{\mathfrak{D}}}}
\newcommand{\qlascoux}{\ensuremath{\overline{\mathfrak{Q}}}}
\newcommand{\schur}{\ensuremath{s}}
\newcommand{\sgroth}{\ensuremath{\overline{s}}}
\newcommand{\qschur}{\ensuremath{S}}
\newcommand{\qgroth}{\ensuremath{\overline{S}}}

\newcommand{\SSF}{\ensuremath{\mathsf{SSF}}}
\newcommand{\ASSF}{\atom\SSF}

\newcommand{\qlascouxSSF}{\qlascoux\SSF}
\newcommand{\LASSF}{\lascouxatom\SSF}
\newcommand{\KSSF}{\key\SSF}
\newcommand{\LPSSF}{\lascoux\SSF}

\newcommand{\str}{\ensuremath{\mathsf{str}}}

\newcommand{\sort}{\ensuremath{\mathtt{sort}}}

\newcommand{\lswap}{\ensuremath{\mathtt{lswap}}}
\newcommand{\Qlswap}{\ensuremath{\mathtt{Qlswap}}}

\newcommand{\latok}{\lascouxatom \ensuremath{2} \kaon}
\newcommand{\qltog}{\qlascoux \ensuremath{2} \glide}

\newcommand{\blank}{\phantom{2}}

\newcommand{\bbb}{\mathsf{b}}

\newcommand{\red}[1]{\textcolor{red}{#1}}

\newcommand{\excise}[1]{}%{$\star$\textsc{#1}$\star$}

\newcommand{\qcolor}[1]{\textcolor{violet}{#1}}
\newcommand{\scolor}[1]{\textcolor{orange}{#1}}
\newcommand{\pcolor}[1]{\textcolor{Green}{#1}}

\newlength\cellsize 

\savebox2{%
\begin{picture}(12,12)
\put(0,0){\line(1,0){12}}
\put(0,0){\line(0,1){12}}
\put(12,0){\line(0,1){12}}
\put(0,12){\line(1,0){12}}
\end{picture}}

\newcommand\cellify[1]{\def\thearg{#1}\def\nothing{}%
\ifx\thearg\nothing\vrule width0pt height\cellsize depth0pt%
  \else\hbox to 0pt{\usebox2\hss}\fi%
  \vbox to 12\unitlength{\vss\hbox to 12\unitlength{\hss$#1$\hss}\vss}}

\newcommand\tableau[1]{\vtop{\let\\=\cr
\setlength\baselineskip{-12000pt}
\setlength\lineskiplimit{12000pt}
\setlength\lineskip{0pt}
\halign{&\cellify{##}\cr#1\crcr}}}

\newcommand\gridify[1]{\vbox to 10\unitlength{\vss\hbox to 10\unitlength{\hss$_{#1}$\hss}\vss}}

\newcommand\pipes[1]{\vtop{\let\\=\cr
\setlength\baselineskip{-10000pt}
\setlength\lineskiplimit{10000pt}
\setlength\lineskip{0pt}
\halign{&\gridify{##}\cr#1\crcr}}}

\usepackage[colorinlistoftodos]{todonotes}

\ytableausetup{centertableaux}

\begin{document}

%%%%%%%%%%%%%%%%%%%%%%%%%%%%%%%%%%%%%%%%%%%%%%%%%%%%%%%%%%%%
%  TITLE PAGE information
%%%%%%%%%%%%%%%%%%%%%%%%%%%%%%%%%%%%%%%%%%%%%%%%%%%%%%%%%%%%

%     [Short Title]{Full Title}
\title{Polynomials from combinatorial $K$-theory}  

%    Information for zeroth author
\author[C. Monical]{Cara Monical}
\address[CM]{Department of Mathematics, University of Illinois at Urbana-Champaign, Urbana, IL 61801, USA}
\email{cmonica2@illinois.edu}
%\thanks{}

\author[O. Pechenik]{Oliver Pechenik}
\address[OP]{Department of Mathematics, University of Michigan, Ann Arbor, MI 48109, USA}
\email{pechenik@umich.edu}
%\thanks{}

\author[D. Searles]{Dominic Searles}
\address[DS]{Department of Mathematics and Statistics, University of Otago, Dunedin 9016, New Zealand}
\email{dominic.searles@otago.ac.nz}
%\thanks{}

%    General info

\date{June 11, 2018}

%\dedicatory{}


\begin{abstract}
We introduce two new bases of the ring of polynomials and study their relations to known bases.  The first basis is the \emph{quasiLascoux} basis, which is simultaneously both a $K$-theoretic deformation of the quasikey basis and also a lift of the $K$-analogue of the quasiSchur basis from quasisymmetric polynomials to general polynomials. We give positive expansions of this quasiLascoux basis into the glide and Lascoux atom bases, as well as a positive expansion of the Lascoux basis into the quasiLascoux basis. As a special case, these expansions give the first proof that the $K$-analogues of quasiSchur polynomials expand positively in multifundamental quasisymmetric polynomials of T.~Lam and P.~Pylyavskyy.

The second new basis is the \emph{kaon} basis, a $K$-theoretic deformation of the fundamental particle basis.  We give positive expansions of the glide and Lascoux atom bases into this kaon basis.

Throughout, we explore how the relationships among these $K$-analogues mirror the relationships among their cohomological counterparts. We make several `alternating sum' conjectures that are suggestive of Euler characteristic calculations.
\end{abstract}

\maketitle
%\tableofcontents

%%%%%%%%%%%%%%%%%%%%%%%%%%%%%%%%%%%%%%%%%%%%%%%%%%%%%%%%%%%%%%%%
%
\section{Introduction}
%
%%%%%%%%%%%%%%%%%%%%%%%%%%%%%%%%%%%%%%%%%%%%%%%%%%%%%%%%%%%%%%%%
\label{sec:introduction}

Let $\Poly_n \coloneqq \mathbb{Z}[x_1, \dots, x_n]$ denote the ring of integral polynomials in $n$ commuting variables. Considerations in representation theory and algebraic geometry give rise to a number of interesting and important bases of $\Poly_n$. This paper contributes two new bases and studies their relations to those bases of established importance; we find that our new bases exhibit well-behaved structure and fill natural holes in the previously developed theory. This study is part of a general program to develop a combinatorial theory of $\Poly_n$ that mirrors the rich classical theory of symmetric functions.

Foremost among known bases of $\Poly_n$ are the celebrated \emph{Schubert polynomials} $\{ \schub_a \}$ of A.~Lascoux and M.-P.~Sch\"{u}tzenberger \cite{Lascoux.Schutzenberger}.
Let $X = \flags_n(\mathbb{C} )$ be the parameter space of complete flags 
\[
0 = V_0 \subset V_1 \subset \dots \subset V_n = \mathbb{C}^n
\]
of nested vector subspaces of $\mathbb{C}^n$, where $\dim V_i = i$. Denote by $\mathsf{B}$ the Borel group of $n \times n$ invertible upper triangular matrices. The standard action of $\mathsf{B}$ on $\mathbb{C}^n$ induces an action on $X$ with finitely-many orbits, whose closures are the \emph{Schubert varieties} of $X$. These subvarieties may be naturally indexed by weak compositions $a= (a_1,\ldots,a_n)$ (i.e., sequences of nonnegative integers) of length $n$ such that $a_i \leq n-i$. The corresponding \emph{Schubert classes} $\{ \sigma_a \}$ form a $\mathbb{Z}$-linear basis for the Chow ring $A^\star(X)$ of subvarieties of $X$ modulo rational equivalence. The Schubert polynomials are polynomial representatives for the Schubert classes in the sense that one has (up to truncation)
\[\schub_a \cdot \schub_b = \sum_c C_{a,b}^c \; \schub_c \qquad \mbox{ if and only if } \qquad \sigma_a \cdot \sigma_b = \sum_c C_{a,b}^c \; \sigma_c.\]

Despite the existence of explicit formulas for Schubert polynomials, it remains a major open problem of algebraic combinatorics to give a positive combinatorial formula for the \emph{Schubert structure constants} $C_{a,b}^c \in \mathbb{Z}_{\geq 0}$.

The (type A) \emph{Demazure characters} $\{ \key_a \}$ of M.~Demazure \cite{Demazure} form another basis of $\Poly_n$, important in representation theory. These are precisely the characters of certain explicitly-defined ${\sf B}$-modules \cite{Demazure, Reiner.Shimozono}. Remarkably, it was shown in \cite{Lascoux.Schutzenberger:key,Reiner.Shimozono} that the Demazure characters \emph{refine} the Schubert polynomials, i.e.,
\[
\schub_a = \sum_b E^a_b \key_b
\]
for some \emph{nonnegative} integers $E^a_b \in \mathbb{Z}_{\geq 0}$.

Letting the symmetric group $S_n$ act on $\Poly_n$ by permuting variables, the $S_n$-invariants are the \emph{symmetric polynomials} $\Sym_n \subset \Poly_n$. Another remarkable property of the bases $\{\schub_a \}$ and $\{ \key_a \}$ of $\Poly_n$ is that each contains (as a subset) the celebrated \emph{Schur basis} $\{ \schur_\lambda \}$ of $\Sym_n$; in fact
\[\{\schub_a \} \cap \Sym_n = \{ \key_a \} \cap \Sym_n = \{ \schur_\lambda \}.\] 
In this sense, both Schubert polynomials and Demazure characters are \emph{lifts} of the Schur basis to the polynomial ring. The Schur basis, moreover, has well-studied and useful refinements into the \emph{quasiSchur polynomials} $\{ \qschur_\alpha \}$ of \cite{HLMvW11:QS, HLMvW11:LRrule} and further into the \emph{fundamental quasisymmetric polynomials} $\{ \fundamental_\alpha \}$ of \cite{Gessel}, both of which are bases of the subspace $\QSym_n \subset \Poly_n$ of quasisymmetric polynomials. (A polynomial $f \in \Poly_n$ is \emph{quasisymmetric} if it is invariant under exchanging $x_i$ and $x_{i+1}$ in those terms of $f$ where both variables do not appear.) 

In general, while a rich combinatorial theory of symmetric and quasisymmetric polynomials has been and continues to be developed, the analogous theory for the full polynomial ring remains relatively sparse. For example, unlike for general Schubert polynomials, several positive combinatorial formulas  (e.g., \cite{Littlewood.Richardson,Knutson.Tao.Woodward, Vakil}) are known for the structure constants of Schur polynomials, i.e., the Littlewood-Richardson coefficients. A natural program, championed by A.~Lascoux \cite{Lascoux:polynomials}, is to develop the analogous combinatorial theory of $\Poly_n$ by
\begin{itemize}
\item lifting known bases and relationships to $\Poly_n$ from the better-understood subrings $\Sym_n$ and $\QSym_n$, and by 
\item developing uniform combinatorial models for these lifted bases and for the relations among them, extending models from $\Sym_n$ and $\QSym_n$. 
\end{itemize} 
The end goal of this program is that this new theory eventually bear dividends on major problems involving polynomials, such as the Schubert problem mentioned above.

Recent work in this area has provided lifts to $\Poly_n$ of the quasiSchur and fundamental bases of $\QSym_n$: respectively, the \emph{quasikey polynomials} $\{ \qkey_a \}$ of \cite{Assaf.Searles:2} and the \emph{(fundamental) slide polynomials} $\{ \slide_a \}$ of \cite{Assaf.Searles}. These families provide further refinements of Schubert polynomials: each Demazure character is a nonnegative combination of quasikeys \cite[Theorem~3.7]{Assaf.Searles:2}, each of which is, in turn, a nonnegative combination of slides \cite[Theorem~3.4]{Assaf.Searles:2}. The slide basis, moreover, like the Schubert basis, has nonnegative structure constants; in fact, unlike the Schubert basis, one even has an analogue of the Littlewood-Richardson rule for multiplying slide polynomials \cite[Theorem~5.11]{Assaf.Searles}.

A more classical approach to the study of Demazure characters is to consider their refinement, not into slides, but rather into the \emph{Demazure atom} basis $\{ \atom_a \}$ of \cite{Lascoux.Schutzenberger:key} (see also, \cite{Mason:atom}). While the Demazure atoms refine the quasikeys \cite[Theorem~3.4]{Searles}, just as slides do, the Demazure atoms have no known direct relation to the slide basis. A common refinement of the Demazure atoms and the slides is provided by the \emph{(fundamental) particle} basis $\{ \particle_a \}$ of \cite{Searles}. The relations among these nine families of polynomials, which we call `cohomological', are illustrated in Figure~\ref{fig:H_bases}.

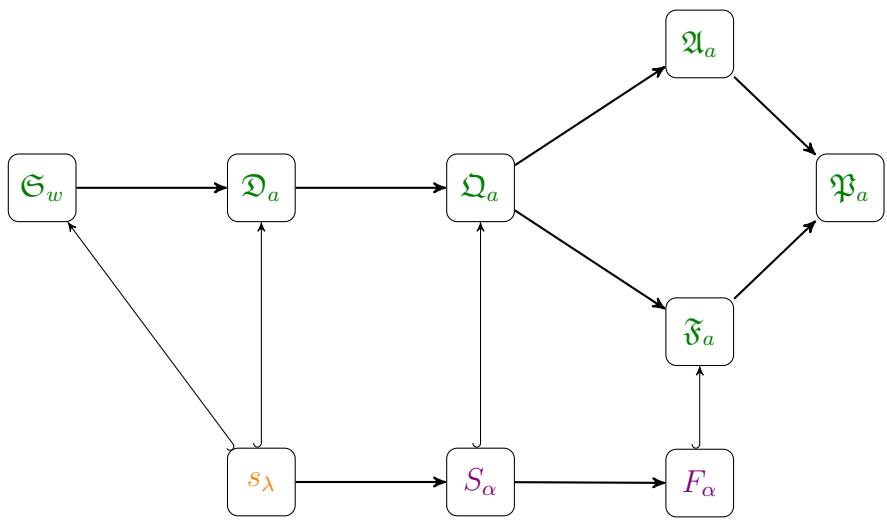
\begin{figure}[h]
\begin{tikzpicture}
\node[Kplaque, rounded corners] (Schubert) {\pcolor{$\schub_w$}};
\node[right=2 of Schubert,Kplaque, rounded corners] (key) {\pcolor{$\key_a$}};
\node[right=2 of key,Kplaque, rounded corners] (qkey) {\pcolor{$\qkey_a$}};
\node[above right=1 and 2 of qkey,Kplaque, rounded corners] (atom) {\pcolor{$\atom_a$}};
\node[below right=1 and 2 of qkey,Kplaque, rounded corners] (slide) {\pcolor{$\slide_a$}};
\node[right=4 of qkey,Kplaque, rounded corners] (pion) {\pcolor{$\mathfrak{P}_a$}};
\node[below=3 of key,Kplaque, rounded corners] (Schur) {\scolor{$\schur_\lambda$}};
\node[below=3 of qkey,Kplaque, rounded corners] (qSchur) {\qcolor{$\qschur_\alpha$}};
\node[below=1.1 of slide,Kplaque, rounded corners] (fundamental) {\qcolor{$F_\alpha$}};
%\node[below right= 5 and 2 of Schubert, Kplaque, rounded corners] (Grothendieck) {\pcolor{$\groth_w$}};
%\node[right=2 of Grothendieck,Kplaque, rounded corners] (Lascoux) {\pcolor{$\lascoux_a$}};
%\node[right=2 of Lascoux,newplaque] (qLascoux) {\pcolor{$\qlascoux_a$}};
%\node[above right=1 and 2 of qLascoux,Kplaque, rounded corners] (Lascouxatom) {\pcolor{$\lascouxatom_a$}};
%\node[below right=1 and 2 of qLascoux,Kplaque, rounded corners] (glide) {\pcolor{$\glide_a$}};
%\node[right=4 of qLascoux,newplaque] (kaon) {\pcolor{$\kaon_a$}};
%\node[below=3 of Lascoux,Kplaque, rounded corners] (symGroth) {\scolor{$Ks_\lambda$}};
%\node[below=3 of qLascoux,Kplaque, rounded corners] (qGroth) {\qcolor{$\qgroth_\alpha$}};
%\node[right=1.88 of qGroth,Kplaque, rounded corners] (multifundamental) {\qcolor{$KF_\alpha$}};
  \begin{scope}[nodes = {draw = none}]
    \path (Schubert) edge[posex]   (key)
    (key) edge[posex]  (qkey)
    (qkey) edge[posex] (atom)
    (qkey) edge[posex] (slide)
    (atom) edge[posex] (pion)
    (slide) edge[posex] (pion)
    (Schur) edge[lift] (Schubert)
    (Schur) edge[lift] (key)
    (qSchur) edge[lift] (qkey)
    (fundamental) edge[lift] (slide)
    (Schur) edge[posex] (qSchur)
    (qSchur) edge[posex] (fundamental)
%    (Grothendieck) edge[posex] (Lascoux)
%    (Lascoux) edge[posex,red]  (qLascoux)
%    (qLascoux) edge[posex,red] (Lascouxatom)
%    (qLascoux) edge[posex,red] (glide)
%    (Lascouxatom) edge[posex,red] (kaon)
%    (glide) edge[posex,red] (kaon)
%    (symGroth) edge[lift] (Grothendieck)
%    (symGroth) edge[lift] (Lascoux)
%    (qGroth) edge[lift] (qLascoux)
%    (multifundamental) edge[lift] (glide)
%    (symGroth) edge[posex] (qGroth)
%    (qGroth) edge[posex,red] (multifundamental)
%    (Schubert) edge[beta] (Grothendieck)
%    (key) edge[beta] (Lascoux)
%    (qkey) edge[beta] (qLascoux)
%    (atom) edge[beta] (Lascouxatom)
%    (slide) edge[beta] (glide)
%    (pion) edge[beta] (kaon)
%    (Schur) edge[beta] (symGroth)
%    (qSchur) edge[beta] (qGroth)
%    (fundamental) edge[beta] (multifundamental)
      ;
  \end{scope}
\end{tikzpicture}
\caption{The nine cohomological families of polynomials considered here. Those depicted in \scolor{orange} are bases of $\Sym_n$, those in \qcolor{purple} are bases of $\QSym_n$, and those in \pcolor{green} are bases of $\Poly_n$. The thinner hooked arrows denote that the basis at the tail is a subset of the basis at the head. The thicker arrows denote that the basis at the head refines the basis at the tail.}
\label{fig:H_bases}
\end{figure}

In this work, we are interested in the \emph{$K$-theoretic analogues} of the bases in Figure~\ref{fig:H_bases}. A major theme of the modern Schubert calculus is the investigation of the geometry of $X = \flags_n$ (and other generalized flag varieties) via richer complex oriented cohomology theories. In the most general such theories, there is ambiguity in the appropriate definition of Schubert classes, as the analogues of the usual push-pull operators fail to satisfy the appropriate braid relations. (For further discussion and partial progress on these problems, see, e.g., \cite{Ganter.Ram,Calmes.Zainoulline.Zhong,Lenart.Zainoulline}.) 

It turns out that this definitional problem is avoided precisely by working in the \emph{connective $K$-theory} (or a specialization thereof) of $X$ \cite{Bressler.Evens}; hence, we restrict ourselves to this context. Complex oriented cohomology theories are determined by their formal group laws, describing how to express the Chern class of a tensor product of two line bundles in terms of the original two Chern classes. For connective $K$-theory, this formal group law is
\begin{equation}\label{eq:fgl}
c_1(L \otimes M) = c_1(L) + c_1(M) + \beta c_1(L) c_1(M),
\end{equation}
where $\beta$ is a formal parameter and $L,M$ are any complex line bundles on $X$. Hence the ordinary cohomology ring is recovered by specializing $\beta = 0$, and the ordinary $K$-theory ring is recovered by specializing $\beta$ to any element of $\mathbb{C}^\star$.

In the connective $K$-theory of $X$, polynomial representatives for a Schubert basis are given by the $\beta$-Grothendieck polynomials $\{ \groth_a \}$ of S.~Fomin and A.~Kirillov \cite{Fomin.Kirillov} (see, \cite{Hudson}). These polynomials form an inhomogeneous basis of $\Poly_n[\beta]$, where $\beta$ is the formal parameter from Equation~(\ref{eq:fgl}). Specializing at $\beta =0$, one recovers the Schubert basis $\{ \schub_a \}$ of $\Poly_n$.  The usual Grothendieck polynomials of A.~Lascoux and M.-P.~Sch\"{u}tzenberger \cite{Lascoux.Schutzenberger} are realized at $\beta = -1$. (To help the reader keep track of the relations between bases, we deviate from established practice by denoting the connective $K$-analogue of each basis of Figure~\ref{fig:H_bases} by merely attaching an `overbar' to the notation for that basis.) 

Intersecting $\{ \groth_a \}$ with $\Sym_n[\beta]$ yields the basis $\{ \sgroth_\lambda \}$ of \emph{symmetric Grothendieck polynomials}. These represent connective $K$-theory Schubert classes on Grassmannians. A number of Littlewood-Richardson rules for $\{ \sgroth_\lambda \}$ are now known (e.g., \cite{Vakil,Thomas.Yong:K,Pechenik.Yong:genomic}), following the first found by A.~Buch \cite{Buch}. Like Schur polynomials, symmetric Grothendieck polynomials have quasisymmetric refinements; each $\sgroth_\lambda$ expands positively in the \emph{quasiGrothendieck} basis $\{ \qgroth_\alpha \}$ of $\QSym_n[\beta]$. This basis, introduced in \cite{Monical}, is the connective $K$-analogue of the quasiSchur basis of $\QSym_n$.

Our first new result is that the quasiGrothendieck basis refines further into the basis of multifundamental quasisymmetric polynomials $\{ \multifundamental_\alpha \}$ of \cite{Lam.Pylyavskyy,Pechenik.Searles}, the connective $K$-analogue of Gessel's fundamental basis of $\QSym_n$.

\begin{theorem}\label{thm:qGroth}
Each quasiGrothendieck polynomial $\qgroth_\alpha \in \QSym[\beta]$ is a positive sum of multifundamental quasisymmetric polynomials. That is,
\[
\qgroth_\alpha = \sum_\gamma J_\gamma^\alpha \; \multifundamental_\gamma,
\]
where $J_\gamma^\alpha \in \mathbb{Z}_{\geq 0}[\beta]$ is a positive polynomial in $\beta$.
\end{theorem}

The multifundamental basis $\{ \multifundamental_\alpha \}$ of $\QSym_n[\beta]$ lifts to the \emph{glide basis} $\{\glide_a\}$ of $\Poly_n[\beta]$  \cite{Pechenik.Searles},  a $\beta$-deformation of the slide polynomials. An analogous deformation of the Demazure characters has been studied in \cite{Lascoux:transition,Ross.Yong,Kirillov:notes,Monical}. We call these the \emph{Lascoux polynomials} $\{ \lascoux_a \}$ in honor of A.~Lascoux, who essentially introduced them. They can be approached via the \emph{Lascoux atom} basis $\{ \lascouxatom_a \}$ of $\Poly_n[\beta]$ \cite{Monical}, a $\beta$-deformation of the Demazure atoms.  Finding a positive formula for the decomposition of Grothendieck polynomials $\{\groth_a\}$ into Lascoux polynomials $\{ \lascoux_a \}$, analogous to that for the decomposition of Schubert polynomials $\{ \schub_a \}$ into Demazure characters $\{ \key_a \}$ is an open problem. (There is an unpublished conjecture for this decomposition by V.~Reiner and A.~Yong; see \cite{Ross.Yong} for discussion.)

Our next major result is to introduce an appropriate $\beta$-deformation of the fundamental particles. The \emph{kaon} basis $\{ \kaon_a \}$ of $\Poly_n[\beta]$ yields a common refinement of the glide and Lascoux atom bases; we give explicit positive formulas for these refinements.

\begin{theorem} \label{thm:kaon_properties}
The set $\{ \kaon_a \}$ of kaons is a basis of $\Poly_n[\beta]$. The kaons deform the fundamental particles, in that specializing $\kaon_a$ at $\beta = 0$ yields the particle $\particle_a$. The kaons are a common refinement of the glide polynomials and of the Lascoux atoms; that is,  
\[\glide_a = \sum_b P_b^a \; \kaon_b \qquad \mbox{and} \qquad \lascouxatom_a = \sum_b Q_b^a \; \kaon_b,\] 
where $P_b^a, Q_b^a \in \mathbb{Z}_{\geq 0}[b]$ are positive polynomials in $\beta$.
\end{theorem}

Finally, our last major result is to introduce the new basis $\{ \qlascoux_a \}$ of \emph{quasiLascoux polynomials}, simultaneously lifting the quasiGrothendieck basis from $\QSym_n[\beta]$ to $\Poly_n[\beta]$ and giving a $\beta$-deformation of the quasikey polynomials. The quasiLascoux polynomials yield a common coarsening of the glide and Lascoux atom bases. We give explicit positive formulas for refining Lascoux polynomials in quasiLascoux polynomials and for refining quasiLascoux polynomials in both glides and Lascoux atoms.

\begin{theorem} \label{thm:qlascoux_properties}
The set $\{ \qlascoux_a \}$ of quasiLascoux polynomials is a basis of $\Poly_n[\beta]$. This basis lifts the quasiGrothendieck basis of $\QSym_n[\beta]$ in that 
\[
\{ \qlascoux_a \} \cap \QSym_n[\beta] = \{ \qgroth_\alpha \}.
\]
The quasiLascoux polynomials deform the quasikeys, in that specializing $\qlascoux_a$ at $\beta = 0$ yields the quasikey $\qkey_a$.
Finally, the quasiLascoux polynomials are a refinement of the Lascoux polynomials and are further refined by the glide polynomials and separately by the Lascoux atoms. That is, 
\[
\lascoux_a = \sum_b L_b^a \; \qlascoux_b, \qquad
\qlascoux_a = \sum_b M_b^a \; \glide_b, \qquad \mbox{and} \qquad 
\qlascoux_a = \sum_b N_b^a \; \lascouxatom_b,\]
%\begin{align*}
%\lascoux_a &= \sum_b L_b^a \; \qlascoux_b, \\
%\qlascoux_a &= \sum_b M_b^a \; \glide_b, \text{ and} \\
%\qlascoux_a &= \sum_b N_b^a \; \lascouxatom_b,
%\end{align*}
where each of $L_b^a, M_b^a, N_b^a \in \mathbb{Z}_{\geq 0}[\beta]$ is a positive polynomial in $\beta$.
\end{theorem}

The relations among these nine families of $K$-theoretic polynomials are illustrated in Figure~\ref{fig:K_bases}; compare to the relations among their $\beta=0$ analogues, as illustrated in Figure~\ref{fig:H_bases}.

\begin{figure}[h]
\begin{tikzpicture}
\node[Kplaque, rounded corners] (Grothendieck) {\pcolor{$\groth_a$}};
\node[right=2 of Grothendieck,Kplaque, rounded corners] (Lascoux) {\pcolor{$\lascoux_a$}};
\node[right=2 of Lascoux,newplaque] (qLascoux) {\pcolor{$\qlascoux_a$}};
\node[above right=1 and 2 of qLascoux,Kplaque, rounded corners] (Lascouxatom) {\pcolor{$\lascouxatom_a$}};
\node[below right=1 and 2 of qLascoux,Kplaque, rounded corners] (glide) {\pcolor{$\glide_a$}};
\node[right=4 of qLascoux,newplaque] (kaon) {\pcolor{$\kaon_a$}};
\node[below=3 of Lascoux,Kplaque, rounded corners] (symGroth) {\scolor{$\sgroth_\lambda$}};
\node[below=2.9 of qLascoux,Kplaque, rounded corners] (qGroth) {\qcolor{$\qgroth_\alpha$}};
\node[below=1.15 of glide,Kplaque, rounded corners] (multifundamental) {\qcolor{$\multifundamental_\alpha$}};
%\node[below right= 5.1 and 2 of Schubert, Kplaque, rounded corners] (Grothendieck) {\pcolor{$\groth_w$}};
%\node[right=2 of Grothendieck,Kplaque, rounded corners] (Lascoux) {\pcolor{$\lascoux_a$}};
%\node[right=2 of Lascoux,newplaque] (qLascoux) {\pcolor{$\qlascoux_a$}};
%\node[above right=1 and 1.9 of qLascoux,Kplaque, rounded corners] (Lascouxatom) {\pcolor{$\lascouxatom_a$}};
%\node[below right=1 and 1.9 of qLascoux,Kplaque, rounded corners] (glide) {\pcolor{$\glide_a$}};
%\node[right=3.6 of qLascoux,newplaque] (kaon) {\pcolor{$\kaon_a$}};
%\node[below=3 of Lascoux,Kplaque, rounded corners] (symGroth) {\scolor{$\sgroth_\lambda$}};
%\node[below=3 of qLascoux,Kplaque, rounded corners] (qGroth) {\qcolor{$\qgroth_\alpha$}};
%\node[right=1.88 of qGroth,Kplaque, rounded corners] (multifundamental) {\qcolor{$\multifundamental_\alpha$}};
  \begin{scope}[nodes = {draw = none}]
    \path %(Schubert) edge[posex]   (key)
    (Grothendieck) edge[posex,dotted] (Lascoux)
    (Lascoux) edge[posex,red]  (qLascoux)
    (qLascoux) edge[posex,red] (Lascouxatom)
    (qLascoux) edge[posex,red] (glide)
    (Lascouxatom) edge[posex,red] (kaon)
    (glide) edge[posex,red] (kaon)
    (symGroth) edge[lift] (Grothendieck)
    (symGroth) edge[lift] (Lascoux)
    (qGroth) edge[lift,red] (qLascoux)
    (multifundamental) edge[lift] (glide)
    (symGroth) edge[posex] (qGroth)
    (qGroth) edge[posex,red] (multifundamental)
      ;
  \end{scope}
%  \begin{scope}[nodes = {draw = none}]
%  \path (qkey) edge[beta] (qLascoux)
%  ;
%  \end{scope}
\end{tikzpicture}
\caption{The $K$-theoretic analogues of the nine cohomological families of polynomials of Figure~\ref{fig:H_bases}.
As in Figure~\ref{fig:H_bases}, families depicted in \scolor{orange} are bases of $\Sym_n$, those in \qcolor{purple} are bases of $\QSym_n$, and those in \pcolor{green} are bases of $\Poly_n$. The thinner hooked arrows denote that the basis at the tail is a subset of the basis at the head. The thicker arrows denote that the basis at the head refines the basis at the tail.
 Those families and arrows that are original to this paper are marked in \textcolor{red}{red}. The dotted arrow is conjectural; see \cite{Ross.Yong}.
}
\label{fig:K_bases}
\end{figure}

Except for the $\beta$-Grothendieck polynomials $\{ \groth_a \}$ and their symmetric subset $\{ \sgroth_\lambda \}$, the geometric significance of these $K$-analogues is currently obscure. While, for example, the glide polynomials seem useful in the study of $\beta$-Grothendieck polynomials (and thereby of the connective $K$-theory of $\flags_n$), it is currently unknown how to interpret any single glide polynomial $\glide_a$ as representing a geometric object or datum. We conclude with some conjectures that appear, to us, to suggest geometric meaning for these polynomials. While it is possible that these conjectures might be proved by entirely combinatorial means (e.g., sign-reversing involutions), they seem to us to have the flavor of Euler characteristic calculations. Ideally, we desire a proof of these conjectures where the relevant polynomials are given appropriate geometric interpretations, so that the coefficients in question become the Euler characteristic of some explicit object.
We note that these conjectures are fundamentally $K$-theoretic, having no analogue in the cohomological setting.

For weak compositions $a$ and $b$, let $M^a_b(\beta)$ denote the coefficient of $\glide_b$ in the glide expansion of $\qlascoux_a$ and let $Q^a_b(\beta)$ denote the coefficient of $\kaon_b$ in the kaon expansion of the Lascoux atom $\lascouxatom_a$. Note that $M^a_b(\beta)$ and $Q^a_b(\beta)$ are nonnegative monomials in the single variable $\beta$.
\begin{conjecture}\label{conj:euler}
Let $a$ be a weak composition. Then we have
\[
\sum_b M^a_b(-1) \in \{ 0, 1 \} \text{ and } \sum_b Q^a_b(-1) \in \{0, 1\},
\]
where both sums are over all weak compositions $b$. \qed
\end{conjecture}

For example, for $a = (0,6,6,2)$, we have \[ \sum_b M_{b}^{a}(\beta) = 16\beta^3 + 75\beta^2 + 94\beta + 36 \]  and 
\[ \sum_b Q^a_b(\beta) = 16\beta^3 + 66\beta^2 + 80\beta + 31. \]
In both cases, substituting $\beta = -1$ yields $1$, as predicted.  
We have verified Conjecture~\ref{conj:euler} by computer for all $a$ with at most 3 zeros and $|a| \leq 7$.

This paper is organized as follows. Section~\ref{sec:defn} recalls the necessary combinatorics of the bases studied in previous works.  Section~\ref{sec:mesons} introduces the kaon basis and proves Theorem~\ref{thm:kaon_properties}, giving the key properties of this basis.  Similarly, Section~\ref{sec:lascoux} introduces the quasiLascoux basis and establishes its key properties via proving Theorems~\ref{thm:qGroth} and~\ref{thm:qlascoux_properties}.  We also suggest there (Conjecture~\ref{conj:multiplying_lascouxs}) a remarkable positivity phenomenon for products of Lascoux polynomials.

\section{Definitions and preliminaries}
\label{sec:defn}

\subsection{Glide polynomials and fundamental slide polynomials}
\label{sec:glide_polynomials}

Given a weak composition $a$, the {\bf positive part} of $a$ is the (strong) composition $a^+$ obtained by deleting all zero terms from $a$. For example, $0102^+ = 12$. 

Given weak compositions $a$ and $b$ of length $n$, say that $b$ {\bf dominates} $a$, denoted by $b \geq a$, if
\begin{equation*}
  b_1 + \cdots + b_i \ge a_1 + \cdots + a_i
\end{equation*}
for all $i=1,\ldots,n$. For example, $0120 \ge 0111$. Note that this partial ordering on weak compositions extends the usual dominance order on partitions.

In \cite{Pechenik.Searles}, a {\bf weak komposition} is defined to be a weak composition where the positive integers may be colored arbitrarily black or \textcolor{red}{red}.
The {\bf excess} $\ex(b)$ of a weak komposition $b$ is the number of red entries in $b$. 

\begin{definition}\cite[Definition~2.2]{Pechenik.Searles}\label{def:glide_komposition}
Let $a$ be a weak composition with nonzero entries in positions $n_1 < \cdots < n_\ell$.
The weak komposition $b$ is a {\bf glide} of $a$ if there exist integers $0 = i_0 < i_1 < \dots < i_\ell$ such that, for each $1 \leq j \leq \ell$, we have
\begin{itemize}
\item[(G.1)] $a_{n_j} = b_{i_{j-1} + 1} + \dots + b_{i_j} - \ex(b_{i_{j-1}+1}, \dots, b_{i_{j}})$,
\item[(G.2)] $i_{j} \leq n_{j}$, and
\item[(G.3)] the leftmost nonzero entry among $b_{i_{j-1} + 1}, \dots, b_{i_j}$  is black.
\end{itemize} 
\end{definition}

Equivalently, a weak komposition $b$ is a glide of the weak composition $a$ if $b$ can be obtained from $a$ by a finite sequence of the following local moves:
\begin{itemize}
\item[(m.1)] $0p \Rightarrow p0$, (for $p \in \mathbb{Z}_{>0}$);
\item[(m.2)] $0p \Rightarrow qr$ (for $p,q,r \in \mathbb{Z}_{>0}$ with $q+r = p)$;
\item[(m.3)] $0p \Rightarrow q\red{r}$ (for $p,q,r \in \mathbb{Z}_{>0}$ with $q+r = p+1$).
\end{itemize}

\begin{example}
Let $a=(0,2,0,0,2,0,1)$. 
The weak kompositions $(1,{\color{red}2},0,2,0,1,{\color{red}1})$ and $(2,1,{\color{red}2},{\color{red}1},{\color{red}1},1,0)$ are glides of $a$.
\end{example}

\begin{definition}\cite[Definition~2.5]{Pechenik.Searles}\label{def:glide_polynomial}
For a weak composition $a$ of length $n$, the {\bf glide polynomial} $\glide^{(\beta)}_a = \glide^{(\beta)}_a(x_1, \dots, x_n)$ is 
\[ \glide^{(\beta)}_a = \sum_b \beta^{\ex(b)} x_1^{b_1} \cdots x_n^{b_n}, \] where the sum is over all weak kompositions $b$ that are glides of $a$. As for $\groth^{(\beta)}_w$, we may drop $\beta$ from the notation, unless it is specialized to a particular value.
\end{definition}

\begin{example}\label{ex:glidepoly} We have
  \[\glide_{0201} = \mathbf{x}^{0201} + \mathbf{x}^{1101} + \mathbf{x}^{0210} + \mathbf{x}^{1110} + \mathbf{x}^{2001} + \mathbf{x}^{2010} + \mathbf{x}^{2100} + \] \[\beta\mathbf{x}^{0211} + \beta\mathbf{x}^{1111} + \beta\mathbf{x}^{1201}+\beta\mathbf{x}^{1210} + \beta\mathbf{x}^{2011} + 2\beta\mathbf{x}^{2101} +  2\beta\mathbf{x}^{2110} + \]
  \[\beta^2\mathbf{x}^{1211} + 2\beta^2\mathbf{x}^{2111},\]
  where $\mathbf{x}^b = x_1^{b_1}\ldots x_n^{b_n}$.
\end{example}

In \cite[Proposition~2.16]{Pechenik.Searles}, it was observed that the {\bf fundamental slide polynomials} $\slide_a$ of \cite{Assaf.Searles} are $\beta=0$ specializations of glide polynomials; that is, 
\[\slide_a = \glide^{(0)}_a.\] We will take this as definitional for fundamental slides.

\subsection{Lascoux atoms and quasiGrothendieck polynomials} 

The \emph{skyline diagram} $D(a)$ of a weak composition $a$ is the diagram with $a_i$ boxes in row $i$, left-justified. In our convention (which is upside-down from that of \cite{HLMvW11:LRrule, Monical} and rotated $90$ degrees counterclockwise from that of \cite{Mason:RSK}), row 1 is the lowest row.
A \emph{triple} of a skyline diagram is a collection of three boxes with two adjacent in a row and either (Type A) the third box is above the right box and the lower row is weakly longer, or (Type B) the third box is below the left box and the higher row is strictly longer. 

\begin{figure}[ht]
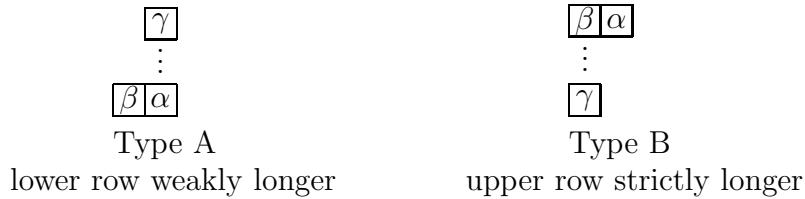

  \begin{displaymath}
    \begin{array}{l}
        \tableau{ & \gamma } \\[-0.5\cellsize]  \hspace{1.25\cellsize} \vdots  \hspace{0.4\cellsize} \\ \tableau{  \beta & \alpha } \\  \mbox{Type A} \\ \hspace{-3\cellsize} \mbox{lower row weakly longer}
    \end{array}
    \hspace{3\cellsize}
    \begin{array}{l}
   \hspace{3\cellsize}   \tableau{ \beta & \alpha } \\[-0.5\cellsize] \hspace{3.35\cellsize}  \vdots \hspace{1.5\cellsize} \\ \hspace{3\cellsize} \tableau{ \gamma & } \\ \hspace{3\cellsize}\mbox{Type B} \\  \mbox{upper row strictly longer}
    \end{array}
  \end{displaymath}
  \caption{\label{fig:triple}Triples for skyline diagrams.}
\end{figure}

Given a filling of the skyline diagram with numbers, a triple (of either type) is called an \emph{inversion triple} if either $\gamma<\alpha\le \beta$ or $\alpha \le \beta < \gamma$, and a \emph{coinversion triple} if $\alpha \le  \gamma \le \beta$. 

In \cite{Monical}, C.~Monical introduced the notion of semistandard set-valued fillings of skyline diagrams, in order to define the (combinatorial) Lascoux atoms, which are $K$-theoretic analogues of the Demazure atoms.   

\begin{definition}\label{def:setSSF}
A {\bf set-valued filling} of a skyline diagram is an assignment a non-empty set of positive integers to each box of the diagram.
The maximum entry in each box is called the {\bf anchor} and all other entries are called {\bf free}.
A set-valued filling is {\bf semistandard} if
\begin{itemize}
\item[(S.1)]  entries do not repeat in a column, 
\item[(S.2)]  rows are weakly decreasing  where sets $A \geq B$ if $\min A \geq \max B$,
\item[(S.3)]  every triple of anchors is an inversion triple, 
\item[(S.4)]  each free entry is in the cell of the least anchor in its column such that (S.2) is not violated, and 
\item[(S.5)] anchors in the first column are equal to their row index.
\end{itemize}
\end{definition}

\begin{remark}\label{rmk:svsky}
The condition (S.5) replaces the equivalent ``basement'' requirement in \cite{Monical}. The condition (S.4) above differs slightly from condition (S4) in \cite[\textsection 1.2]{Monical}, which puts free entries in the lowest possible row such that (S.2) is not violated. These definitions are however equivalent in the sense that there is a simple weight-preserving (and moreover column set-preserving) bijection between these two notions of semistandard set-valued skyline fillings via rearranging the free entries appropriately in each column. The convention in Definition~\ref{def:setSSF} turns out to be more natural in the context of the operations we wish to perform on these fillings.
\end{remark}

Given a set-valued filling $F$ of shape $a$, we define the \emph{weight} of $F$ to be the weak composition $\wt(F) = (c_1,\ldots,c_n)$ where $c_i$ is the number of $i$'s in $F$.  Furthermore, $|F| = |\wt(F)|.$  Likewise, the \emph{excess} of $F$, denoted $\ex(F)$, is the number of free entries of $F$, or equivalently, $\ex(F) = |F|-|a|$. 
Given a weak composition $a$, let $\LASSF(a)$ be the set of semistandard set-valued skyline diagrams of shape $a$. See Figure~\ref{fig:SetSkyFill} for examples; the anchor entries are given in bold.

\begin{figure}[ht]
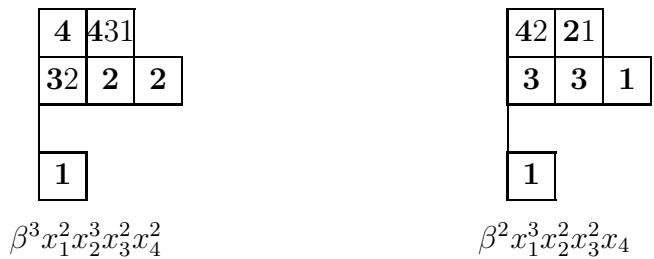

\begin{center}
\begin{minipage}{.4\linewidth}
\begin{center}
	\begin{ytableau}
	\none & \textbf{4} & \textbf{4}31 \\
	\none  & \textbf{3}2 & \textbf{2} & \textbf{2}\\
	\none \vline \\
	\none  & \textbf{1}  
	\end{ytableau} \\\vspace*{.1in}
	$\beta^3x_1^2x_2^3x_3^2x_4^2$
\end{center}
\end{minipage}
\begin{minipage}{.4\linewidth}
\begin{center}
	\begin{ytableau}
	 \none & \textbf{4}2 & \textbf{2}1 \\ 
	 \none & \textbf{3} & \textbf{3} & \textbf{1}\\ 
	 \none \vline \\
	\none & \textbf{1} \\ 
	\end{ytableau} \\\vspace*{.1in}
$\beta^2x_1^3x_2^2x_3^2x_4$
\end{center}
\end{minipage}
\end{center}
\caption{Two elements of $\LASSF(1,0,3,2)$ and their corresponding monomial weights. The anchor in each box is drawn in bold; the unbolded entries are free.}\label{fig:SetSkyFill}
\end{figure}

\begin{definition}[\cite{Monical}]  Given a weak composition $a$, the {\bf (combinatorial) Lascoux atom} $\lascouxatom_a^{(\beta)}$ is 
\[ \lascouxatom_a^{(\beta)} = \sum_{F \in \LASSF(a)} \beta^{\ex(F)}\x^{\wt(F)}. \] 
\label{defn:lasAtom}
\end{definition} 

We will drop the word ``combinatorial'' in Definition~\ref{defn:lasAtom} from now on, as we will not have cause to consider any of the other conjecturally equivalent definitions of Lascoux atoms.

\begin{definition}\label{def:Demazure_atom}
For a weak composition $a$, the {\bf Demazure atom} $\atom_a$ is the $\beta = 0$ specialization of the corresponding Lascoux atom. That is,
\[
\atom_a = \lascouxatom_a^{(0)}.
\] 
Equivalently, $\atom_a$ is the generating polynomial for semistandard set-valued skyline diagrams of shape $a$ where exactly one number appears in each box.
\end{definition}

Demazure atoms were originally defined by A.~Lascoux and M.-P.~Sch\"utzenberger \cite{Lascoux.Schutzenberger:key} in a different way; the equivalence of Definition~\ref{def:Demazure_atom} is due to work by S.~Mason \cite{Mason:atom}.
Thus, $\lascouxatom_a$ is a inhomogeneous deformation of $\atom_a$.  The Lascoux atoms form a (finite) basis of $\Poly[\beta] = \mathbb{Z}[x_1,x_2,...][\beta]$ by \cite[Proposition~2.2]{Monical}.

\begin{definition}[{\cite[\textsection 3]{Monical}}]\label{def:QuasiGroth}
Given a (strong) composition $\alpha$, the {\bf quasiGrothendieck polynomial} $\qgroth_\alpha^{(\beta)}$ in $n$ variables is  
\[\qgroth_\alpha^{(\beta)}(x_1,\ldots , x_n) = \sum_{a^+ = \alpha} \lascouxatom_a, \]
where the sum is over weak compositions of length $n$.
The $\beta=0$ specialization 
\[\qgroth_\alpha^{(0)}(x_1, \ldots , x_n)=\qschur_\alpha(x_1, \ldots , x_n)\]
is the {\bf quasiSchur polynomial} $\qschur_\alpha$ of J.~Haglund, K.~Luoto, S.~Mason, and S.~van Willigenburg  \cite{HLMvW11:LRrule}.
\end{definition}

In \cite{Monical}, it was shown that the quasiGrothendieck polynomials form another finite basis of $\QSym[\beta]$.

The Lascoux atoms refine the symmetric Grothendieck polynomials:

\begin{theorem}[\cite{Monical}]\label{thm:GLasDecomp}
\[ \sgroth_\lambda^{(\beta)}(x_1, \ldots x_n) = \sum_{\sort(a) = \lambda} \lascouxatom_a^{(\beta)},\]
where the sum is over weak compositions of length $n$ and $\sort(a)$ is the partition formed by sorting the parts of $a$ in weakly decreasing order. 
\end{theorem}

Combining Theorem~\ref{thm:GLasDecomp} and Definition~\ref{def:QuasiGroth} yields the decomposition of $\sgroth_\lambda^{(\beta)}$ into quasiGrothendieck polynomials:

\begin{corollary}[\cite{Monical}] \label{cor:grothtoquasigroth}
\[ \sgroth_\lambda^{(\beta)}(x_1, \ldots , x_n) = \sum_{\sort(\alpha) = \lambda} \qgroth_\alpha^{(\beta)}(x_1,\ldots , x_n).\]
\end{corollary}

Setting $\beta=0$ in Theorem~\ref{thm:GLasDecomp} and Corollary~\ref{cor:grothtoquasigroth} recovers  the earlier refinements by J.~Haglund, K.~Luoto, S.~Mason, and S.~van Willigenburg \cite{HLMvW11:QS} of Schur polynomials respectively into  Demazure atoms and into quasiSchur polynomials.

\section{The mesonic bases and their relations}\label{sec:mesons}
In this section, we introduce a new ``kaon'' basis of polynomials. These new polynomials are a simultaneous refinement of both glide polynomials and Lascoux atoms. 
In Section~\ref{sec:lascoux} we will introduce a new ``quasiLascoux'' basis of polynomials, which are a simultaneous coarsening of these two bases. We will moreover find that quasiLascoux polynomials stand in the same relation to Lascoux atoms as glide polynomials do to kaons. See Figure~\ref{fig:K_bases} for a visual representation of the relationships among these various bases.

\subsection{Kaons}

\begin{definition}\label{def:mesonic_glides}
Let $a$ be a weak composition with nonzero entries in positions $n_1 < \dots < n_\ell$. 
The weak komposition $b$ is a {\bf mesonic glide} of $a$ if, for each $1 \leq j \leq \ell$, we have
\begin{itemize}
\item[(G.$1'$)] $a_{n_j} = b_{n_{j-1} + 1} + \dots + b_{n_j} - \ex(b_{n_{j-1}+1}, \dots, b_{n_{j}})$,
\item[(G.$3'$)] the leftmost nonzero entry among $b_{n_{j-1} + 1}, \dots, b_{n_j}$  is black, and
\item[(G.$4'$)] $b_{n_j} \neq 0$.
\end{itemize} 

Equivalently, a weak komposition $b$ is a mesonic glide of $a$ if $b$ can be obtained from $a$ by a finite sequence of the local moves (m.1), (m.2), and (m.3) that never applies (m.1) at positions $n_j-1$ and $n_j$ for any $j$.
\end{definition}

Observe that, in particular, a mesonic glide is a glide that happens to satisfy additional conditions.

\begin{example}\label{ex:mesonic_glide}
Let $a = (0,3,0,2)$. Then $b = (2,1,1,\red{2})$ is a mesonic glide of $a$. 

On the other hand, while $b' = (3,1,0,\red{2})$ is also a glide of $a$, it is not mesonic. To see this fact, observe that $a_{n_1} = a_2 = 3$, while 
\[
b'_{n_{j-1} + 1} + \dots + b'_{n_j} - \ex(b'_{n_{j-1}+1}, \dots, b'_{n_{j}}) = b'_1 + b'_2 - \ex(b'_1, b'_2) = 3 +1 - 0 = 4,
\]
in violation of (G.$1'$).

The reader may check that both $b$ and $b'$ can be obtained from $a$ by a finite sequence of the local moves (m.1), (m.2), and (m.3). However, the reader may also check that $b'$ cannot be so obtained without applying (m.1) at positions $1$ and $2$.
\end{example}

\begin{definition}
Let $a$ be a weak composition. The {\bf kaon} $\kaon_a^{(\beta)}$ is the following generating function for mesonic glides:
\[
\kaon_a^{(\beta)} \coloneqq \sum_b \beta^{\ex(b)} {\bf x}^b,
\]
where the sum is over all mesonic glides of $a$.
\end{definition}

\begin{example}\label{ex:kaon}
Let $a = (0,3,0,2)$. Then the corresponding kaon is
\begin{align*}
\kaon_a^{(\beta)} &= \x^{0302} + \x^{0311} + \x^{1202} + \x^{1211} + \x^{2102} + \x^{2111} \\
 &+ \beta\x^{0312} + \beta\x^{0321} + \beta\x^{1212} + \beta\x^{1221} + \beta\x^{1302} + \beta\x^{1311} \\ 
 &+  \beta\x^{2112} + \beta\x^{2121} + \beta\x^{2202} + \beta\x^{2211} + \beta\x^{3102} + \beta\x^{3111} \\
 &+ \beta^2\x^{1312} + \beta^2\x^{1321} + \beta^2\x^{2212} + \beta^2\x^{2221} + \beta^2\x^{3112} + \beta^2\x^{3121}.
\end{align*}
The reader may enjoy realizing each exponent vector as a mesonic glide of $a$.

Although this example is multiplicity-free, in general kaons have nontrivial coefficients in their monomial expansions. For example, the kaon $\kaon_{002}$ contains the monomial $\beta \x^{111}$ with coefficient $2$, corresponding to the distinct mesonic glides $(1,\red{1},1)$ and $(1,1,\red{1})$ of $(0,0,2)$.
\end{example}

\subsection{Fundamental properties of kaons and the kaon expansion of glide polynomials}

Every glide polynomial $\glide_a$ is a positive sum of kaons.
\begin{proposition}\label{prop:glide2kaon}
For any weak composition $a$, we have
\[
\glide_a^{(\beta)} = \sum_{\substack{b \geq a \\ b^+= a^+ }} \kaon_b^{(\beta)}.
\]
\end{proposition}
\begin{proof}
Let $a$ be a weak composition with nonzero entries in positions $n_1 < \dots < n_\ell$.
Suppose $g$ is a glide of $a$. Then there are $0 = i_0 < i_1 < \cdots < i_\ell$ satisfying conditions (G.1), (G.2), and (G.3) of Definition~\ref{def:glide_komposition}. Then $g$ may be obtained from $a$ via a $2$-step process. First, apply (m.1) repeatedly to move each nonzero entry of $a$ from position $n_j$ to position $i_j$. Call the resulting weak composition $b$. Note that $b$ satisfies $b \geq a$ and $b^+=a^+$. Second, apply some sequence of (m.1), (m.2) and (m.3) to obtain the weak komposition $g$ from $b$. In this second step, note that we never apply (m.1) at positions $i_j - 1$ and $i_j$ for any $j$. 

Hence every glide $g$ of $a$ is a mesonic glide of a weak composition $b$ with $b \geq a$ and $b^+= a^+$, and so every term of the left-hand polynomial is a term of the right-hand polynomial. 

Conversely, every mesonic glide of such a weak composition $b$ with $b \geq a$ and $b^+= a^+$ is clearly a glide of $a$. 
Thus to complete the proof, we only need to show that for every glide $g$ of $a$ there is at most one $b$ with $b^+ = a^+$ such that $g$ is a mesonic glide of $b$.
Let $\alpha = a^+ = (\alpha_1,\ldots,\alpha_\ell).$ Suppose that $b,c$ are weak compositions with $b^+ = c^+ = \alpha$ such that $g$ is a mesonic glide of both $b$ and $c$. Say $b$ has nonzero entries in positions $s_1 < \dots < s_\ell$ while $c$ has nonzero entries in positions $t_1 < \dots < t_\ell$.
By the definition of mesonic glide, we know that  \[ \alpha_j = b_{s_j} = g_{s_{j-1}+1} + \ldots + g_{s_j} - {\sf ex}(g_{s_{j-1}+1},\ldots,g_{s_j}), \] the leftmost nonzero entry among $g_{s_{j-1}+1}, \ldots, g_{s_j}$ is black, and $g_{s_j} \neq 0$. In the same way, we have that
\[ \alpha_j = c_{t_j} = g_{t_{j-1}+1} + \ldots + g_{t_j} - {\sf ex}(g_{t_{j-1}+1},\ldots,g_{t_j}), \] the leftmost nonzero entry among $g_{t_{j-1}+1}, \ldots, g_{t_j}$ is black, and $g_{t_j} \neq 0$.
If $s_j = t_j$ for all $j$, then $b= c$ and we are done. Otherwise, there is a least index $i$ such that $s_i \neq t_i$.
Without loss of generality, assume $s_i > t_i$.
Then, 
\begin{align*}
\alpha_i &= g_{s_{i-1}+1} + \dots + g_{s_i} - {\sf ex}(g_{s_{i-1}+1}, \dots, g_{s_i}) \\
&= g_{s_{i-1}+1} + \dots + g_{t_i} + g_{t_i + 1} + \dots +  g_{s_i} - {\sf ex}(g_{s_{i-1}+1}, \dots, g_{t_i}) - {\sf ex}(g_{t_i+1},\dots,g_{s_i}) \\
&=  g_{t_{i-1}+1} + \dots + g_{t_i} - {\sf ex}(g_{t_{i-1}+1}, \dots, g_{t_i}) +  g_{t_i + 1} + \dots +  g_{s_i} - {\sf ex}(g_{t_i+1},\dots,g_{s_i}) \\
&= \alpha_i +   g_{t_i + 1} + \dots +  g_{s_i} - {\sf ex}(g_{t_i+1},\dots,g_{s_i}),
\end{align*}
and so we have \[ 0 = g_{t_i + 1} + \dots +  g_{s_i} - {\sf ex}(g_{t_i+1},\dots,g_{s_i}). \]
This is only possible if each of $g_{t_i +1},\dots, g_{s_i}$ is either 0 or a red \red{1}.
Since $g_{s_i} \neq 0$, there is at least one red \red{1} in this set of entries.
However, the first nonzero entry of $g_{t_i+1},\dots,g_{t_{i+1}}$ is required to be black, a contradiction.
\end{proof}

\begin{theorem}\label{thm:kaon_basis}
The set 
\[
\{ \beta^k \kaon^{(\beta)}_a : k \in \mathbb{Z}_{\geq 0} \text{ and $a$ is a weak composition of length $n$} \} 
\]
is an additive basis of the free $\mathbb{Z}$-module $\Poly_n[\beta]$. Hence, for any fixed $p \in \mathbb{Z}$, 
\[ 
\{ \kaon^{(p)}_a : \text{$a$ is a weak composition of length $n$} \}
\]
is a basis of $\Poly_n$.
\end{theorem}
\begin{proof}
By Proposition~\ref{prop:glide2kaon}, every glide polynomial can be written as a positive sum of kaons, and indeed the transition matrix is unitriangular with respect to the  lexicographic total order on weak compositions. Hence, the transition matrix is invertible over $\mathbb{Z}$, and the theorem follows from the fact that glide polynomials are an additive basis of $\Poly_n[\beta]$, as shown in \cite[Theorem~2.6]{Pechenik.Searles}.
\end{proof}

A homogeneous basis $\particle_a$ of $\Poly_n$ called \emph{fundamental particles} was introduced in \cite{Searles}. This basis is a common refinement of fundamental slides and Demazure atoms. We will show that the kaon basis plays the analogous role for glide polynomials and Lascoux atoms.

\begin{proposition}\label{prop:pions_are_specialized_kaons}
The fundamental particles $\particle_a$ are the $\beta=0$ specialization of kaons:  \[\particle_a = \kaon_a^{(0)}.\]
\end{proposition}
\begin{proof}
This is clear from the definitions of the two families of polynomials.
\end{proof}

\begin{remark}
Proposition~\ref{prop:pions_are_specialized_kaons} motivates our choice of the name `kaon' for these polynomials. In high energy physics, the `$K$' fundamental particles are the \emph{K-mesons} or \emph{kaons}. By analogy, perhaps, the cohomological specialization $\pi_a = \kaon_a^{(0)}$ could be named for the lighter analogue of the $K$-meson, the \emph{$\pi$-meson} or \emph{pion}. These mesons play a role in the structural integrity of atomic nuclei; somewhat analogously, we will show momentarily that the kaon polynomials decompose the Lascoux atoms, thereby controlling their structure.
\end{remark}

\begin{remark}
In light of Proposition~\ref{prop:pions_are_specialized_kaons}, setting $\beta = 0$ in Proposition~\ref{prop:glide2kaon} recovers \cite[Proposition~4.7]{Searles} on the expansion of fundamental slide polynomials into fundamental particles. Note that although the $K$-theoretic deformations of both families of polynomials are significantly larger, the matrix of basis change from Proposition~\ref{prop:glide2kaon} is exactly the same as for fundamental slides into fundamental particles. Taking $p = 0$ in Theorem~\ref{thm:kaon_basis} recovers \cite[Proposition~4.6]{Searles} as a special case.
\end{remark}

The kaon basis does not have positive structure coefficients. Nonetheless, we conjecture the following:

\begin{conjecture}\label{conj:kaon_pos_product}
For any weak compositions $a$ and $b$, the product
\[ \kaon_a \cdot \glide_b \]
of a kaon and a glide polynomial expands positively in the kaon basis.
\end{conjecture}

For example, we have
\begin{align*}
\kaon_{(2,0,1)} \cdot \glide_{(1,0,2)} = &\kaon_{(3,0,3)} + \beta\kaon_{(3,1,3)} + \beta\kaon_{(3,2,2)} \\
&+ \beta^2 \kaon_{(3,2,3)} + \beta^2 \kaon_{(3,3,2)}. 
\end{align*} 
We have computationally verified Conjecture~\ref{conj:kaon_pos_product} for all weak compositions $a,b$ with at most 3 zeros and $|a|, |b| \leq 5$. To our knowledge, Conjecture~\ref{conj:kaon_pos_product} is new even in the special $\beta=0$ case of the fundamental particle expansion of the product of a fundamental particle by a fundamental slide polynomial.

\subsection{The kaon expansion of a Lascoux atom}

\begin{definition}\label{def:meson_highest}
Let $a$ be a weak composition and $T \in \LASSF(a)$. We say $T$ is {\bf meson-highest} if, for every integer $i$ appearing in $T$, either 
\begin{itemize}
\item the leftmost $i$ is in the leftmost column and is an anchor, or 
\item there is a $i^\uparrow$ in some column weakly to the right of the leftmost $i$ and in a different box, where $i^\uparrow$ is the smallest label greater than $i$ appearing in $T$.
\end{itemize}
In light of the following Theorem~\ref{thm:LascouxAtom2kaon}, we write $\latok(a)$ for the set of all meson-highest $T \in \LASSF(a)$.
\end{definition}

\begin{theorem}\label{thm:LascouxAtom2kaon}
For any weak composition $a$, we have
\begin{equation}\label{eq:LascouxAtom2kaon}
\lascouxatom_a^{(\beta)} = \sum_{T \in \latok(a)} \beta^{|T|-|a|}\kaon_{\wt(T)}^{(\beta)}.
\end{equation}
In particular, every Lascoux atom $\lascouxatom_a$ is a positive sum of kaons.
\end{theorem}

To prove Theorem~\ref{thm:LascouxAtom2kaon}, we must first develop properties of a {\bf destandardization} map on $\LASSF(a)$. Fix $T \in \LASSF(a)$. Consider the least integer $i$ with the property that 
\begin{itemize}
\item the leftmost $i$ in $T$ is not an anchor in the leftmost column, and
\item this leftmost $i$ has no $i^\uparrow$ weakly to its right in a different box;
\end{itemize} replace every $i$ in $T$ with an $i + 1$. (If this results in two instances of $i+1$ in a single box, delete one.) Repeat this replacement process until
no further replacements can be made: the final result is the destandardization $\destand(T)$. (This algorithm necessarily terminates, as we only perform replacement on labels $i$ that are strictly less than the maximum entry $k$ of $T$; this is because $k$ is guaranteed to appear as an anchor in the leftmost column of $T$.) For an example of these notions, see Figure~\ref{fig:lascouxatom}.

\begin{figure}[ht]
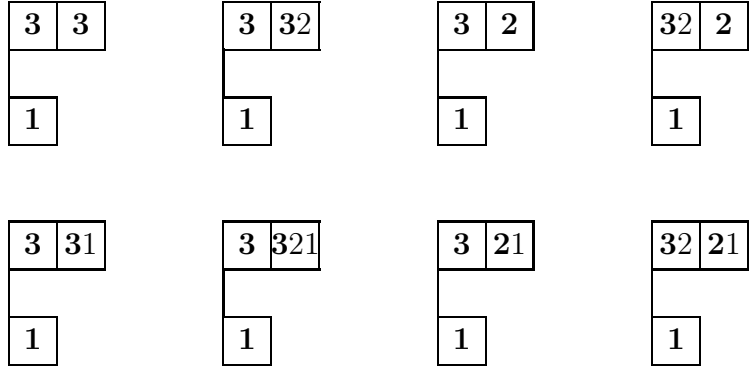

  \begin{center}
    \begin{displaymath}
      \begin{array}{c@{\hskip2\cellsize}c@{\hskip2\cellsize}c@{\hskip2\cellsize}c@{\hskip2\cellsize}c}
	\begin{ytableau}
	\none  & \textbf{3} & \textbf{3} \\
	\none \vline \\
	\none  & \textbf{1}  
	\end{ytableau} \vspace{1cm} &
	\begin{ytableau}
	\none  & \textbf{3} & \textbf{3}2 \\
	\none \vline \\
	\none  & \textbf{1}  
	\end{ytableau} 	&

	\begin{ytableau}
	\none  & \textbf{3} & \textbf{2} \\
	\none \vline \\
	\none  & \textbf{1}  
	\end{ytableau} 	&
	\begin{ytableau}
	\none  & \textbf{3}2 & \textbf{2} \\
	\none \vline \\
	\none  & \textbf{1}  
	\end{ytableau} 	\\
	\begin{ytableau}
	\none  & \textbf{3} & \textbf{3}1 \\
	\none \vline \\
	\none  & \textbf{1}  
	\end{ytableau} 	&
	\begin{ytableau}
	\none  & \textbf{3} & \textbf{3}21 \\
	\none \vline \\
	\none  & \textbf{1}  
	\end{ytableau} &

	\begin{ytableau}
	\none  & \textbf{3} & \textbf{2}1 \\
	\none \vline \\
	\none  & \textbf{1}  
	\end{ytableau} 	&

	\begin{ytableau}
	\none  & \textbf{3}2 & \textbf{2}1 \\
	\none \vline \\
	\none  & \textbf{1}  
	\end{ytableau} 	&
      \end{array}
    \end{displaymath}
    \caption{The eight elements of $\LASSF(102)$. The set $\latok(102)$ consists of two fillings, specifically the leftmost filling in each of the two rows.  Each of the eight illustrated fillings destandardizes to the leftmost filling in its row.}\label{fig:lascouxatom}
  \end{center}
\end{figure}

\begin{remark}
In fact, the order in which we perform replacements does not affect the resulting destandardization. Nonetheless, it is convenient to fix the explicit replacement order chosen above.
\end{remark}

\begin{lemma}\label{lem:destand_of_LATabs}
Let $a$ be a weak composition.
If $T \in \LASSF(a)$, then $\destand(T) \in \latok(a)$. Moreover, destandardization is a retraction onto $\latok(a)$, as we have $\destand(T) = T$ if and only if $T \in \latok(a) \subseteq \LASSF(a)$.
\end{lemma}
\begin{proof}
Fix the weak composition $a$. By definition, if $T \in \latok(a)$, then $\destand(T) = T$. Moreover, if $T \notin \latok(a)$, then by definition $\destand(T) \neq T$. Hence, the third sentence of the lemma is clear.

It remains to establish the second sentence of the lemma, so fix $T \in \LASSF(a)$. It is enough to show that $\destand(T) \in \LASSF(a)$, for then $\destand(T) \in \latok(a)$ follows easily, as the destandardization algorithm does not terminate until the extra conditions defining $\latok(a)$ as a subset of $\LASSF(a)$ are satisfied. Indeed, since destandardization is defined as a sequence of replacements, it is enough by induction to show that any single such replacement produces an element of $\LASSF(a)$. 

Suppose we apply replacement to the letters $i$ and the result is $T'$. Then, by assumption, the leftmost $i \in T$
\begin{itemize}
\item is not an anchor in the leftmost column of $T$, and
\item does not have an $i^\uparrow$ weakly to its right in $T$ and in a different box.
\end{itemize}
We want to show that $T'$ satisfies the conditions (S.1)--(S.5).

\smallskip
\noindent
{\bf (S.1):}
If there is no column of $T$ containing both $i$ and $i+1$, then it is clear that $T'$ satisfies (S.1). Hence, suppose column $c$ of $T$ contains both $i$ and $i+1$. Then, $i^\uparrow = i+1$. Since $T$ has then no $i+1$ weakly to the right of the leftmost $i$, $c$ must be the column of the leftmost $i$. Moreover, $i$ and $i+1$ must appear in the same box $\bbb$ of column $c$ in $T$.
Thus, replacement results in two instances of $i+1$ in $\bbb$, one of which we then delete by construction. Thus, $T'$ has no repeated entries in any column.

\smallskip
\noindent
{\bf (S.2):}
If row $i$ of $T$ contains an entry $i$, then by (S.2) and (S.5) for $T$, row $i$ has $i$ as an anchor in the first column. Thus, in this case, the leftmost $i$ is an anchor in the first column, contradicting our assumptions on the number $i$.

Therefore by (S.2) and (S.5) for $T$, every entry $i$ in $T$ is in a row with an index $j$ strictly greater than $i$. Moreover, for each such $j > i$, we have by (S.2) for $T$ that all labels
strictly left of the leftmost $i$ in row $j$ are strictly greater than $i$. Hence, replacing every $i$ in $T$ with
with $i + 1$ preserves the rows being weakly decreasing.

\smallskip
\noindent
{\bf (S.3):}
To see that no type A coinversion triples appear in $T'$, suppose first that $T$ has a type A inversion triple
with $\gamma < \alpha \leq \beta$. This could become a coinversion triple in $T'$ only if $\gamma = i$ and $\alpha = i +1 = i^\uparrow$. However, in this case, $T$ has $i$ and $i^\uparrow$ in distinct boxes of the same column, contradicting our assumptions on the number $i$.

Now, suppose instead that $T$ has a type A inversion triple with $\alpha \leq \beta < \gamma$. This could become become a coinversion triple in $T'$ only if $\gamma = i+1 = i^\uparrow$ and $\beta = i$. However, in this case, $T$ has $i+1$ appearing strictly right of $i$, again contradicting our assumptions on the number $i$.

To see that no type B coinversion triples appear in $T'$, suppose first that $T$ has a type B inversion triple
with $\gamma < \alpha \leq \beta$. This could become a coinversion triple in $T'$ only
if $\gamma = i$ and $\alpha =  i + 1 = i^\uparrow$. 
However, then $T$ would have an $i^\uparrow$ strictly right of an $i$, contradicting our assumptions on $i$. 

Finally, suppose $T$ has a type B inversion triple with $\alpha \leq \beta < \gamma$. This could become a coinversion triple in $T'$ only if $\gamma =  i+1 = i^\uparrow$ and $\beta = i$. However, then $T$ would have an $i$ and an $i+1$ in distinct boxes of the same column, again contradicting our assumptions on $i$. 

\smallskip
\noindent
{\bf (S.4):}
If a free entry $i$ of $T$ becomes a free entry $i+1$ of $T'$ and is not deleted, then its anchor entry $j$ is larger than $i+1$ in both $T$ and $T'$. In particular, since $j$ was the smallest anchor entry in this column accepting a free entry $i$ in $T$ (by (S.4) for $T$), $j$ is still the smallest anchor entry accepting a free entry $i+1$ in $T'$.

If an anchor entry $i$ of $T$ becomes an anchor entry $i+1$ of $T'$, then since any other anchor entry in this column is either greater than $i+1$ or smaller than $i$, any free entries in the cell of this anchor entry are still with the smallest possible anchor entry. Any other free entries in a column where an anchor entry $i$ becomes an anchor entry $i+1$ are also still with the smallest possible anchor entry, again since other anchor entries in this column are either greater than $i+1$ or smaller than $i$.

\smallskip
\noindent
{\bf (S.5):} By construction, the replacement operation taking $T$ to $T'$ does not affect the anchor entries in the first column.
\end{proof}

\begin{lemma}\label{lem:kaon_for_destandardized_wt}
Let $a$ be a weak composition and $S \in \latok(a)$. Then 
\[
\kaon_{\wt(S)} = \sum_{T \in \destand^{-1}(S)} \beta^{|T|-|S|} \x^{\wt(T)}.
\]
\end{lemma}
\begin{proof}
We must establish a weight-preserving bijection between mesonic glides of $\wt(S)$ and fillings $T \in \destand^{-1}(S)$.

Fix $T \in \destand^{-1}(S)$. Define the {\bf colored weight} $\kwt(T)$ of $T$ to be the weak komposition obtained by coloring the $(i+1)$st entry of $\wt(T)$ \red{red} if any $i+1$ is deleted after replacing every $i$ with $i+1$ during a step of destandardization. (Note that at most one $i+1$ can be so deleted, as otherwise before replacement there would be two boxes, each containing both $i$ and $i+1$; in such a case, there would be an $i+1$ strictly right of an $i$, in violation of the rules for when to perform replacement.)

We claim that $\kwt(T)$ is a mesonic glide of $\wt(S)$. Consider the process of destandardization of $T$ to produce $S$. Each time we replace every $i$ in $T'$ by $i+1$ to produce $T''$, we change the $i$th and $(i+1)$st entries of the colored weight by 
\[(q,\red{r}) \mapsto (0, q+r-1),\] if a duplicate $i+1$ is deleted, or 
\[(q,r) \mapsto (0, q+r),\]
 if not. Since these are the inverses of the local move (m.3) in the first case and either (m.1) or (m.2) in the second case, it follows then that $\kwt(T)$ is a glide of the weak composition $\wt(S)$. 
 
Let $\wt(S)$ have nonzero entries in positions $n_1 < \dots < n_\ell$.  A local change to the colored weight
\[
(q,0) \mapsto (0,q)
\]
in positions $n_j -1$ and $n_j$ for some $j$ would correspond to a step of destandardization replacing every instance of $n_j-1$ with $n_j$ in a $T'$ that contains no label $n_j$. Since $S$ is meson-highest, the entries $n_j$ of $S$ satisfy the meson-highest condition. Since $T'$ destandardizes to $S$, then entries $n_j -1$ of $T'$ therefore also satisfy the meson-highest condition. This contradicts the application of such a destandardization step to $T'$.
Hence, $\kwt(T)$ is a mesonic glide of $\wt(S)$.

For the other direction, let $b$ be a mesonic glide of the weak composition $\wt(S)$. We construct the unique $T\in \destand^{-1}(S)$ such that $\kwt(T)=b$ as follows. Suppose $\wt(S)$ has nonzero entries in positions $n_1< \cdots < n_\ell$.

Begin with the smallest entries of $S$, i.e., the $\wt(S)_{n_1}$ instances of the entry $n_1$ in $S$. Consider the sequence $s_1 = (1, \dots, i_1)$ of positions of entries associated (via (G.1)) to $n_1$ in the mesonic glide $b$. From this sequence, we construct a string $\str_1$ by appending letters to the empty word as follows. 
Reading $s_1$ from left to right, for each $k \in s_1$, consider the entry $b_k$. If $b_k$ is a black entry, append $b_k$ black $k$'s to the end of the string; if $b_k$ is a red entry, append one red $\red{k}$ followed by $b_k-1$ black $k$'s.
For example, if $\wt(S)= (0,0,0,0, 4, 3, \ldots)$ and $b = (1,0,{\color{red}2},0,2, 3,\ldots )$, then $\str_{1} = 1{\color{red}3}355$.  

Now replace the instances of $n_1$ in $S$ with entries of $\str_1$ as follows.
In the rightmost box of $S$ containing an $n_1$, replace that $n_1$ with the first letter of $\str_1$. If the next letter of $\str_1$ is \red{red}, place it in the same box the previous entry was placed in. Otherwise, place it in the next box of $S$ to the left that contains an $n_1$, replacing that $n_1$. Continue in this manner until all letters of $\str_1$ have been placed into $S$. This procedure is well-defined since by (S.1) no more than one entry $n_1$ appears in any column of $S$; and since by (G.1) the number $\wt(S)_{n_1}$ of $n_1$'s in $S$ is exactly the length of the string $\str_1$ minus the number of red entries, each black entry goes in a different box of $S$, and each red entry goes in a box along with a black entry. Repeat this process with $n_2, \ldots , n_\ell$.   
(This algorithm is a minor modification of that appearing in \cite[Proof of Theorem~3.5]{Pechenik.Searles}.)

For example, let $S$ be the leftmost filling in the lower row of Figure~\ref{fig:lascouxatom}. We have $\wt(S)=(2,0,2)$. A mesonic glide of $\wt(S)$ is $b=(2,1,{\color{red}2})$. We have $n_1=1$, $n_2=3$, $\str_1 = 11$ and $\str_2 = 2{\color{red}3}3$. We construct the filling $T$ from $S$ and $b$ as follows. Replace the two $1$s in $S$ with the two $1$s from $\str_1$ (which does nothing to $S$), then replace the two $3$s in $S$ with $\str_2$. The $2$ and ${\color{red}3}$ from $\str_2$ are placed in the rightmost box (along with the $1$ already there); then, the remaining $3$ from $\str_2$ is placed in the box to the left. In this way, we obtain $T$ as the second filling from the left in the lower row of Figure~\ref{fig:lascouxatom}. Note that indeed $T$ has weight $(2,1,2)$ and destandardizes to $S$. 

By construction, the resulting filling $T$ has weight $b$ and destandardizes to $S$. We need to show that $T \in \LASSF(a)$. 
To see this, notice that the entries of $\str_j$ (which replace the entries $n_j$ in $S$) are all strictly larger than $n_{j-1}$ and weakly smaller than $n_j$. This fact implies that all inequalities between entries of boxes in $S$ are preserved, and thus all of (S.1), (S.2), (S.3) and (S.4) are preserved. Finally, since $S\in \latok(a) \subseteq \LASSF(a)$, the first column of $S$ has an anchor $a_i$ in each nonempty row $a_i$, and so these (nonzero) $a_i$'s are a subset of the $n_i$'s. Since $b$ is a mesonic glide of $\wt(S)$, by definition the last entry of each $\str_i$ is $n_i$. Therefore, the process of constructing $T$ from $b$ and $S$ ensures that the anchors in the first column are replaced by themselves, i.e., they do not change. Hence (S.5) is also satisfied, and we have $T \in \LASSF(a)$.

The uniqueness of $T$ follows from the lack of choice at each step in this process.
\end{proof}

\begin{proof}[Proof of Theorem~\ref{thm:LascouxAtom2kaon}]
For $S \in \latok(a)$,  Lemma~\ref{lem:kaon_for_destandardized_wt} says that
\[
\kaon_{\wt(S)} = \sum_{T \in \destand^{-1}(S)} \beta^{|T|-|S|}\x^{\wt(T)}.
\]
Therefore,
\begin{align*}
\sum_{S \in \latok(a)} \beta^{|S|-|a|} \kaon_{\wt(S)} &= \sum_{S \in \latok(a)} \beta^{|S|-|a|} \sum_{T \in \destand^{-1}(S)} \beta^{|T|-|S|} \x^{\wt(T)} \\&= \sum_{U \in \LASSF(a)} \beta^{|U|-|a|}\x^{\wt(U)} \\&= \lascouxatom_a,
\end{align*}
where the second equality is by Lemma~\ref{lem:destand_of_LATabs} and the third equality is by definition.
\end{proof}

\begin{remark}
Setting $\beta=0$ in Theorem~\ref{thm:LascouxAtom2kaon} recovers the expansion of Demazure atoms into fundamental particles from \cite[Theorem~4.17]{Searles}. In particular, the meson-highest fillings with no free entries are exactly the ``particle-highest'' fillings of \cite{Searles}.
\end{remark}

\section{Lascoux and quasiLascoux polynomials}\label{sec:lascoux}

\subsection{QuasiLascoux polynomials}
The quasikey basis of \cite{Assaf.Searles:2} is a common coarsening of the fundamental slide and Demazure atom bases of $\Poly_n$, a refinement of the basis of Demazure characters, and a lifting of the quasiSchur basis of \cite{HLMvW11:QS} from $\QSym_n$ to $\Poly_n$. 
We introduce a $K$-theoretic analogue of the quasikey basis---alternatively, a lifting of the quasiGrothendieck basis from $\QSym[\beta]$ to $\Poly[\beta]$. See Figure~\ref{fig:K_bases} for a visual representation of the relationships among these various bases.

\begin{definition}\label{def:QL}
Given a weak composition $a$, the {\bf quasiLascoux polynomial} $\qlascoux_a$ is given by
\[\qlascoux_a^{(\beta)} =  \sum_{\stackrel{b\ge a}{b^+=a^+}}\lascouxatom_b^{(\beta)}.\]
\end{definition}

\begin{proposition}\label{prop:qlascoux_basis}
The set 
\[
\{ \beta^k \qlascoux^{(\beta)}_a : k \in \mathbb{Z}_{\geq 0} \text{ and $a$ is a weak composition of length $n$} \} 
\]
is an additive basis of the free $\mathbb{Z}$-module $\Poly_n[\beta]$. Hence, for any fixed $p \in \mathbb{Z}$, 
\[ 
\{ \qlascoux^{(p)}_a : \text{$a$ is a weak composition of length $n$} \}
\]
is a basis of $\Poly_n$.
\end{proposition}
\begin{proof}
First, we show that 
\[
\{ \beta^k \qlascoux^{(\beta)}_a : k \in \mathbb{Z}_{\geq 0} \text{ and $a$ is a weak composition of length $n$} \} 
\] is a spanning set.
A monomial $m$ in $\Poly_n[\beta]$ corresponds to a pair $(k,a)$, where $k \in \mathbb{Z}_{\geq 0}$ records the power of $\beta$ in $m$ and  $a$ is the weak composition of length $n$ recording the degrees of $x_1, x_2, \dots, x_n$ in $m$. Write $\mathcal{M}$ for the set of all such pairs.

For $a$ a weak composition of length $n$, let $s_a$ denote the string $1^{a_1}2^{a_2} \cdots n^{a_n}$. Write $M(a)$ for the largest element of $a$ and write $\ell_+(a)$ for the position of the rightmost nonzero entry of $a$. Define a total order on $\mathcal{M}$ by $(k,a) \succ (h,b)$ if 
\begin{itemize}
\item $\ell_+(a) > \ell_+(b)$;
\item $\ell_+(a) = \ell_+(b)$ and $M(a) > M(b)$;
\item $\ell_+(a) = \ell_+(b)$, $M(a) = M(b)$, and $s_a >_{\rm lex} s_b$; or
\item $a = b$ and $k > h$.
\end{itemize}
Now, observe that the $\prec$-leading term of $\qlascoux_a$ is $\beta^0 \x^a$. Hence, if the $\prec$-leading term of $f \in \Poly_n[\beta]$ is $c_a \beta^k \x^a$, then the $\prec$-leading term of 
\[
f_1 \coloneqq f - c_a \beta^k \qlascoux_a
\]
is $c_b \beta^h \x^b$ for some $(b, h) \prec (a,k)$. Then,
\[
f_2 \coloneqq f - c_b \beta^h \qlascoux_b
\]
has $\prec$-leading term $c_d \beta^j \x^d$ for some $(d, j) \prec (b,h)$, etc. Since $\prec$ is a well order on $\mathcal{M}$, this process must terminate. Hence, $f$ is a finite $\mathbb{Z}$-linear combination of elements of our putative basis.

Linear independence is immediate from each element of the putative basis having a different leading term. This proves the first sentence of the proposition.

The second sentence of the proposition is immediate from the first.
\end{proof}

\begin{proposition}\label{prop:generalize_qkey}
The quasikey polynomials are the $\beta=0$ specialization of quasiLascoux polynomials: \[\qlascoux_a^{(0)}=\qkey_a\]
\end{proposition}
\begin{proof}
In \cite{Searles}, it is proved that 
\[\qkey_a =  \sum_{\stackrel{b\ge a}{b^+=a^+}}\atom_b.\] 
The statement then follows from Definition~\ref{def:QL} and the fact that $\lascouxatom_b^{(0)} = \atom_b$.
\end{proof}

The following is clear from Definitions~\ref{def:QuasiGroth}  and~\ref{def:QL}.

\begin{proposition}\label{prop:generalize_qGroth}
Suppose that the positions of the nonzero entries in the weak composition $a$ form an interval and that $a_k$ is the last nonzero entry of $a$. Then, 
\[  \qlascoux_a^{(\beta)}  = \qgroth_{a^+}^{(\beta)}(x_1, \ldots , x_k).   \] 
In particular, every quasiGrothendieck polynomial is a quasiLascoux polynomial.
\end{proposition}

Moreover, we have

\begin{proposition}\label{prop:qlascoux_stable_limit}
Let $a$ be a weak composition. Then the stable limit $\lim_{m \to \infty} \qlascoux_{0^m \times a}^{(\beta)}$  of the quasiLascoux polynomial $\qlascoux_a^{(\beta)}$ is the quasiGrothendieck function $\qgroth_{a^+}^{(\beta)}(x_1, x_2, \ldots ).$
\end{proposition}
\begin{proof}
Let $m>0$ and consider the polynomial $\qlascoux_{0^m\times a}^{(\beta)}(x_1, \ldots , x_{m})$. Observe that, for any weak composition $b$, the Lascoux atom $\lascouxatom_b^{(\beta)}$ is divisible by $x_{b_i}$ whenever $b_i>0$. Hence, if $\lascouxatom_b^{(\beta)}$ appears in the Lascoux atom expansion of $\qlascoux_{0^m\times a}^{(\beta)}$, then it is annihilated on restriction to $m$ variables unless $\ell_+(b) \leq m$. Thus, by Definitions~\ref{def:QuasiGroth}  and~\ref{def:QL},
\[\qlascoux_{0^m\times a}^{(\beta)}(x_1, \ldots , x_m)  =  \sum_{\substack{b^+=a^+ \\ \ell_+(b)\le m}}  \lascouxatom_b^{(\beta)} =  \qgroth_{a^+}^{(\beta)}(x_1, \ldots , x_m).\]
The proposition then follows by letting $m\to \infty$.
\end{proof}

\begin{remark}
Setting $\beta=0$ in Proposition~\ref{prop:qlascoux_stable_limit} gives a new proof of the fact that quasikey polynomials stabilise to quasiSchur functions; this was proved via a different method in \cite[\textsection 4.3]{Assaf.Searles:2}.
\end{remark}

To give the monomial expansion of a quasiLascoux polynomial directly, we define $\qlascouxSSF(a)$ to be all set-valued skyline fillings of shape $a$ satisfying (S.1)--(S.4), as well as
\begin{itemize}
\item[(S.$5'$)] anchors in the first column are at most their row index and decrease from top to bottom.
\end{itemize}
We call $\qlascouxSSF(a)$  the {\bf set-valued quasi-skyline fillings} of shape $a$. Then we have 

\begin{proposition}\label{prop:LSetSF}
Given a weak composition $a$, we have
\[\qlascoux_a^{(\beta)} = \sum_{S\in \qlascouxSSF(a)}\beta^{|S|-|a|}\x^{\wt(S)}.\]
\end{proposition}
\begin{proof}
There is a weight-preserving bijection
\[\qlascouxSSF(a) \,\,\,\, \longleftrightarrow \coprod_{\substack{b\ge a \\ b^+ = a^+}}  \LASSF(b),\] 
where the image of $T\in \qlascouxSSF(a)$ is obtained by moving each row of $T$ downwards until each anchor in the first column is equal to its row index. This is well-defined since moving rows without changing their relative order does not affect the inversion/coinversion status of any triple.
The proposition then follows from Definitions~\ref{defn:lasAtom} and~\ref{def:QL}.
\end{proof}

\begin{example}
For $a = (1,0,2)$, the set $\qlascouxSSF(a)$ consists of the ten fillings shown in Figure~\ref{fig:qlascoux}. 
Therefore, we have 
\[\qlascoux_{102}^{(\beta)} = \x^{102} + \x^{111} + \x^{120} + \beta\x^{112} + \beta\x^{202} + \beta\x^{121} + \beta\x^{211} + \beta\x^{220} + \beta^2\x^{212} + \beta^2\x^{221}.\]
\end{example}

\begin{figure}[h]
  \begin{center}
    \begin{displaymath}
      \begin{array}{c@{\hskip2\cellsize}c@{\hskip2\cellsize}c@{\hskip2\cellsize}c@{\hskip2\cellsize}c@{\hskip2\cellsize}c}
	\begin{ytableau}
	\none  & \textbf{3} & \textbf{3} \\
	\none \vline \\
	\none  & \textbf{1}  
	\end{ytableau} \vspace{1cm} &
	\begin{ytableau}
	\none  & \textbf{3} & \textbf{3}2 \\
	\none \vline \\
	\none  & \textbf{1}  
	\end{ytableau} 	&
	\begin{ytableau}
	\none  & \textbf{3} & \textbf{3}1 \\
	\none \vline \\
	\none  & \textbf{1}  
	\end{ytableau} 	&
	\begin{ytableau}
	\none  & \textbf{3} & \textbf{3}21 \\
	\none \vline \\
	\none  & \textbf{1}  
	\end{ytableau} 	&
	\begin{ytableau}
	\none  & \textbf{3} & \textbf{2} \\
	\none \vline \\
	\none  & \textbf{1}  
	\end{ytableau} 	\\
	\begin{ytableau}
	\none  & \textbf{3} & \textbf{2}1 \\
	\none \vline \\
	\none  & \textbf{1}  
	\end{ytableau} 	&
	\begin{ytableau}
	\none  & \textbf{3}2 & \textbf{2} \\
	\none \vline \\
	\none  & \textbf{1}  
	\end{ytableau} 	&
	\begin{ytableau}
	\none  & \textbf{3}2 & \textbf{2}1 \\
	\none \vline \\
	\none  & \textbf{1}  
	\end{ytableau} 	&
	\begin{ytableau}
	\none  & \textbf{2} & \textbf{2} \\
	\none \vline \\
	\none  & \textbf{1}  
	\end{ytableau} 	&
	\begin{ytableau}
	\none  & \textbf{2} & \textbf{2}1 \\
	\none \vline \\
	\none  & \textbf{1}  
	\end{ytableau} 	&								
      \end{array}
    \end{displaymath}
    \caption{\label{fig:qlascoux}The ten elements of $\qlascouxSSF(102)$.}
  \end{center}
\end{figure}

\subsection{The glide expansion of a quasiLascoux polynomial}

\begin{definition}\label{def:QY_setvalued_skylines}
Let $a$ be a weak composition and let $S \in \qlascouxSSF(a)$ be a set-valued quasi-skyline filling. We say $S$ is {\bf quasiYamanouchi} if, for every integer $i$ appearing in $S$, either 
\begin{itemize}
\item the leftmost $i$ is an anchor in row $i$ of the leftmost column, or 
\item there is an $i+1$ in some column weakly right of the leftmost $i$ and in a different box.
\end{itemize}
In light of the following Theorem~\ref{thm:QuasiLascoux2glide}, we write $\qltog(a)$ for the set of all quasiYamanouchi $S \in \qlascouxSSF(a)$.
\end{definition}

\begin{remark}
The quasiYamanouchi condition (Definition~\ref{def:QY_setvalued_skylines}) for set-valued quasi-skyline fillings is exactly the meson-highest condition (Definition~\ref{def:meson_highest}) for semistandard set-valued skyline fillings, with $i^\uparrow$ replaced by $i+1$.
\end{remark}

\begin{example}
The first and third fillings in the top row of Figure~\ref{fig:qlascoux} are quasiYamanouchi. The other fillings in Figure~\ref{fig:qlascoux} are not.
\end{example}

\begin{theorem}\label{thm:QuasiLascoux2glide}
For any weak composition $a$, we have
\begin{equation*}
\qlascoux_a^{(\beta)} = \sum_{S \in \qltog(a)} \beta^{|S|-|a|}\glide_{\wt(S)}^{(\beta)}.
\end{equation*}
In particular, every quasiLascoux polynomial $\qlascoux_a^{(\beta)}$ is a positive sum of glide polynomials.
\end{theorem}

To prove Theorem~\ref{thm:QuasiLascoux2glide}, we introduce a destandardization map $\destand_Q$ on $\qlascouxSSF(a)$. Fix $T\in \qlascouxSSF(a)$. Consider the least integer $i$ appearing in $T$ with the property that

\begin{itemize}
\item the leftmost $i$ in $T$ is not an anchor in the leftmost column, and
\item it has no $i+1$ weakly to its right in a different box;
\end{itemize}
replace every $i$ in $T$ with an $i+1$. (If this results in two instances of $i+1$ in a single box, delete one.) Repeat this replacement process until no further replacements can be made; the final result is the destandardization $\destand_Q(T)$.

\begin{remark}
The destandardization map $\destand_Q$ is exactly the destandardization map $\destand$ of Section~\ref{sec:mesons} with $i+1$ everywhere in place of $i^\uparrow$. 
\end{remark}

\begin{example}
The first, second, fifth, seventh and ninth fillings of Figure~\ref{fig:qlascoux} destandardize to the first filling; the remaining fillings destandardize to the third filling.
\end{example}

The following result is entirely analogous to Lemma~\ref{lem:destand_of_LATabs}.

\begin{lemma}\label{lem:destand_of_qlascouxSSFs}
Let $a$ be a weak composition. If $T\in \qlascouxSSF(a)$, then $\destand_Q(T) \in \qltog(a)$. Moreover, destandardization is a retraction onto $\qltog(a)$, as we have $\destand_Q(T)=T$ if and only if $T\in \qltog(a) \subseteq \qlascouxSSF(a)$. 
\end{lemma}
\begin{proof}
Identical to the proof of Lemma~\ref{lem:destand_of_LATabs}, with $i+1$ everywhere in place of $i^\uparrow$. 
\end{proof}

\begin{lemma}\label{lem:glide_for_destandardized_wt}
Let $a$ be a weak composition and $S\in \qltog(a)$. Then, 
\[\glide_{\wt(S)} = \sum_{T\in \destand_Q^{-1}(S)} \beta^{|T|-|S|}\x^{\wt(T)}.\]
\end{lemma}
\begin{proof}
We need to establish a weight-preserving bijection between the glides of the weak composition $\wt(S)$ and the fillings $T\in \destand_Q^{-1}(S)$.

Fix $T\in \destand_Q^{-1}(S)$. Define the colored weight $\kwt(T)$ of $T$ to be the weak komposition obtained by coloring the $(i+1)$st entry of $\wt(T)$ \red{red} if an $i+1$ is deleted after replacing every $i$ with an $i+1$ during a step of destandardization. By the same exact reasoning as in the analogous step of the proof of Lemma~\ref{lem:kaon_for_destandardized_wt}, we have that $\kwt(T)$ is a glide of the weak composition $\wt(S)$. 

For the other direction, given a glide $b$ of the weak composition $\wt(S)$,  we must construct (the unique) $T\in \destand_Q^{-1}(S)$ such that $\wt(T)=b$. This is achieved by the same process, and via the same argument, as in Lemma~\ref{lem:kaon_for_destandardized_wt}.
\end{proof}

\begin{proof}[Proof of Theorem~\ref{thm:QuasiLascoux2glide}]
For $S \in \qltog(a)$,  Lemma~\ref{lem:glide_for_destandardized_wt} says that
\[
\glide_{\wt(S)}^{(\beta)} = \sum_{T \in \destand_Q^{-1}(S)} \beta^{|T|-|S|} \x^{\wt(T)}.
\]
Therefore,
\begin{align*}
\sum_{S \in \qltog(a)} \beta^{|S|-|a|}\glide_{\wt(S)}^{(\beta)} &= \sum_{S \in \qltog(a)} \beta^{|S|-|a|}\sum_{T \in \destand_Q^{-1}(S)} \beta^{|T|-|S|} \x^{\wt(T)} \\&= \sum_{U \in \qlascouxSSF(a)} \beta^{|U| - |a|} \x^{\wt(U)} \\&= \qlascoux_a^{(\beta)},
\end{align*}
where the second equality is by Lemma~\ref{lem:destand_of_qlascouxSSFs} and the third equality is Definition~\ref{def:QL}.
\end{proof}

\begin{remark}
Setting $\beta=0$ in the statement of Theorem~\ref{thm:QuasiLascoux2glide} yields a positive combinatorial formula for the fundamental slide expansion of a quasikey polynomial in terms of quasiYamanouchi semi-skyline fillings. Such a formula was alluded to in \cite{Searles}, but not stated explicitly.
\end{remark}

\begin{corollary}\label{cor:qG_into_multifunds}
The quasiGrothendieck polynomials expand positively in the basis of multi-fundamental quasisymmetric polynomials.
\end{corollary}
\begin{proof}
By Theorem~\ref{thm:QuasiLascoux2glide}, any quasiLascoux polynomial expands positively in the glide basis. By Proposition~\ref{prop:generalize_qGroth}, the quasiGrothendieck polynomials are included among the quasiLascoux polynomials. The statement then follows from the fact (\cite[\textsection 3.2]{Pechenik.Searles}) that the quasisymmetric glide polynomials are the multi-fundamental quasisymmetric polynomials and form a basis of $\QSym[\beta]$.
\end{proof}

\begin{remark}
The number of terms in the expansion of Corollary~\ref{cor:qG_into_multifunds} generally grows without bound as the number of variables increases. In the limit, one finds that a quasiGrothendieck function is a positive sum of multi-fundamental quasisymmetric functions, but that this sum of formal power series has infinitely-many terms.
\end{remark}

\subsection{Lascoux polynomials}
In this section, we study the combinatorial Lascoux polynomials of \cite[\textsection 5]{Monical} and their relations to the other families of polynomials discussed in this paper.
Given a skyline diagram, we augment it on the left with an additional column $0$, called the {\bf basement}. We write $b_i$ for the entry in row $i$ of the basement. Expanding on Definition~\ref{def:setSSF} and Remark~\ref{rmk:svsky}, a set-valued skyline filling with basement is {\bf semistandard} if it (including the basement) satisfies (S.1), (S.2), (S.3), and (S.4). Basement entries do not count towards the weight $\wt(F)$ of a filling $F$ with basement. In diagrams, we shade the boxes of the basement in \textcolor{gray}{gray} to distinguish them from the ordinary boxes.

For $a$ a weak composition, let $\overleftarrow{a}$ denote the weak composition formed by reversing the order of the parts of $a$. For example, if $a=(0,1,0,3)$, then $\overleftarrow{a} = (3,0,1,0)$. 
 Let $\LPSSF(a)$ be the set of semistandard set-valued skyline fillings of shape $\overleftarrow{a}$ with basement $b_i = n-i+1$.

\begin{definition}[{\cite[Equation~(2.3)]{HLMvW11:LRrule}, \cite[\textsection 5]{Monical}}]\label{def:L}
Let $a$ be a weak composition.  The {\bf (combinatorial) Lascoux polynomial} $\lascoux_a$ is given by \[ \lascoux_a^{(\beta)} = \sum_{F \in \LPSSF(a)} \beta^{\ex(F)}\x^{\wt(F)}. \] The {\bf Demazure character} is the $\beta=0$ specialization of the corresponding Lascoux polynomial: 
\[
\key_a = \lascoux_a^{(0)}.
\]
\end{definition}

Demazure characters are, in fact, characters of certain modules with relation to Schubert calculus \cite{Demazure}; no such representation-theoretic realization of Lascoux polynomials is currently known.

\begin{figure}[ht]
\begin{center}
    \begin{displaymath}
      \begin{array}{c@{\hskip2\cellsize}c@{\hskip2\cellsize}c@{\hskip2\cellsize}c@{\hskip2\cellsize}c@{\hskip2\cellsize}c}
\begin{ytableau}
*(gray) 1 & \textbf{1} \\
*(gray) 2 \\
*(gray) 3 & \textbf{2} & \textbf{1} \\
\end{ytableau} \vspace*{.5cm} & 
\begin{ytableau}
*(gray) 1 & \textbf{1} \\
*(gray) 2 \\
*(gray) 3 & \textbf{2} & \textbf{2} \\
\end{ytableau} &
\begin{ytableau}
*(gray) 1 & \textbf{1} \\
*(gray) 2 \\
*(gray) 3 & \textbf{2} & \textbf{2}1 \\
\end{ytableau} &
\begin{ytableau}
*(gray) 1 & \textbf{1} \\
*(gray) 2 \\
*(gray) 3 & \textbf{3} & \textbf{1} \\
\end{ytableau} &
\begin{ytableau}
*(gray) 1 & \textbf{1} \\
*(gray) 2 \\
*(gray) 3 & \textbf{3} & \textbf{2} \\
\end{ytableau} \\
\begin{ytableau}
*(gray) 1 & \textbf{1} \\
*(gray) 2 \\
*(gray) 3 & \textbf{3} & \textbf{3} \\
\end{ytableau} \vspace*{.5cm} &
\begin{ytableau}
*(gray) 1 & \textbf{1} \\
*(gray) 2 \\
*(gray) 3 & \textbf{3} & \textbf{2}1 \\
\end{ytableau} &
\begin{ytableau}
*(gray) 1 & \textbf{1} \\
*(gray) 2 \\
*(gray) 3 & \textbf{3} & \textbf{3}1 \\
\end{ytableau} &
\begin{ytableau}
*(gray) 1 & \textbf{1} \\
*(gray) 2 \\
*(gray) 3 & \textbf{3} & \textbf{3}2 \\
\end{ytableau} &
\begin{ytableau}
*(gray) 1 & \textbf{1} \\
*(gray) 2 \\
*(gray) 3 & \textbf{3} & \textbf{3}21 \\
\end{ytableau} \\ &
\begin{ytableau}
*(gray) 1 & \textbf{1} \\
*(gray) 2 \\
*(gray) 3 & \textbf{3}2 & \textbf{1} \\
\end{ytableau} &
\begin{ytableau}
*(gray) 1 & \textbf{1} \\
*(gray) 2 \\
*(gray) 3 & \textbf{3}2 & \textbf{2} \\
\end{ytableau} &
\begin{ytableau}
*(gray) 1 & \textbf{1} \\
*(gray) 2 \\
*(gray) 3 & \textbf{3}2 & \textbf{2}1 \\
\end{ytableau} &
      \end{array}
    \end{displaymath}
    \caption{The 13 elements of $\LPSSF(102)$.}\label{fig:lascouxPoly}
  \end{center}
\end{figure}

\begin{example}
The fillings in Figure~\ref{fig:lascouxPoly} show that the monomial expansion of the Lascoux polynomial $\lascoux_{102}^{(\beta)}$ is
\begin{align*}
       \lascoux_{102}^{(\beta)} &= x_1^2x_2 + x_1x_2^2 + \beta x_1^2x_2^2 + x_1^2x_3 + x_1x_2x_3 \\ 
       &+ x_1x_3^2 + \beta x_1^2x_2x_3  + \beta x_1^2x_3^2 + \beta x_1x_2x_3^2 + \beta^2 x_1^2x_2x_3^2 \\
        &+ \beta x_1^2 x_2x_3 + \beta x_1x_2^2x_3 + \beta^2 x_1^2x_2^2x_3. \end{align*} The Demazure character $\key_{102}$ is given by setting $\beta=0$ in this expression; equivalently, $\key_{102}$ is the weight generating function for those five fillings in Figure~\ref{fig:lascouxPoly} that contain only one number in each box.
\end{example}

The main result of the remainder of this paper is to show that every Lascoux polynomial $\lascoux_a^{(\beta)}$ is a positive sum of quasiLascoux polynomials. Towards this goal, we first show the weaker result that $\lascoux_a^{(\beta)}$ is a positive sum of Lascoux atoms.

Given a weak composition $a$, let $\sort(a)$ be the rearrangement of the parts of $a$ into weakly decreasing order, and let $w(a)$ be the minimal (Coxeter) length permutation sending $a$ to $\sort(a)$. 

\begin{theorem}\label{thm:lDecomp}
For any weak composition $a$, we have 
\[
\lascoux_a^{(\beta)} = \sum_{\stackrel{\sort(b)=\sort(a)}{w(b)\le w(a)}} \lascouxatom_b^{(\beta)},
\]
where  $\le$ denotes the strong Bruhat order on permutations. 
In particular, every Lascoux polynomial $\lascoux_a^{(\beta)}$ is a positive sum of Lascoux atoms.
\end{theorem}

\begin{example} Let $a=(0,1,0,3)$. Then $\sort(a)=(3,1,0,0)$ and $w(a)=3241$. Hence,
\[\lascoux_{0103}^{(\beta)} = \lascouxatom_{0103}^{(\beta)} + \lascouxatom_{1003}^{(\beta)} + \lascouxatom_{0130}^{(\beta)} + \lascouxatom_{1030}^{(\beta)}+ \lascouxatom_{1300}^{(\beta)} + \lascouxatom_{0301}^{(\beta)} + \lascouxatom_{0310}^{(\beta)} + \lascouxatom_{3001}^{(\beta)}+ \lascouxatom_{3010}^{(\beta)} + \lascouxatom_{3100}^{(\beta)}.\]
\end{example}

Specializing Theorem~\ref{thm:lDecomp} at $\beta=0$ recovers a known formula for the Demazure atom expansion of a Demazure character (see, e.g., \cite{Mason:atom}, \cite{HLMvW11:LRrule}).

\begin{remark}
In \cite{Lascoux:transition}, A.~Lascoux introduced $K$-theoretic analogues of Demazure characters in terms of divided difference operators. C.~Monical conjectured (\cite[Conjecture~5.3]{Monical}) that such an \emph{operator Lascoux polynomial} $\lascoux_a^{{\rm op}, (\beta)}$ equals the corresponding combinatorial Lascoux polynomial $\lascoux_a^{(\beta)}$ of Definition~\ref{def:L}. Similarly, there is an \emph{operator Lascoux atom} $\lascouxatom_a^{{\rm op}, (\beta)}$ such that conjecturally $\lascouxatom_a^{{\rm op}, (\beta)} = \lascouxatom_a^{(\beta)}$ (\cite[Conjecture~5.2]{Monical}). (See, \cite[\textsection 5]{Monical} for details.)
By \cite[Theorem~5.1]{Monical}, these operator Lascoux polynomials expand into operator Lascoux atoms according to the same combinatorial formula as in Theorem~\ref{thm:lDecomp}. That is, for any weak composition $a$, we have 
\[
\lascoux_a^{{\rm op},(\beta)} = \sum_{\stackrel{\sort(b)=\sort(a)}{w(b)\le w(a)}} \lascouxatom_b^{{\rm op},(\beta)}.
\]
Hence, Theorem~\ref{thm:lDecomp} proves the equivalence of \cite[Conjecture~5.2]{Monical} and \cite[Conjecture~5.3]{Monical}.   
\end{remark}

Before continuing with our proofs of the positive expansions of a Lascoux polynomial in the Lascoux atom and quasiLascoux bases, we formulate the following related conjecture, which is in some sense a strengthening of Theorem~\ref{thm:lDecomp}.

\begin{conjecture}\label{conj:multiplying_lascouxs}
Let $a$ and $b$ be weak compositions. Then $\lascoux_a^{(\beta)} \cdot \lascoux_b^{(\beta)}$ is a positive sum of Lascoux atoms.
\end{conjecture}

For example, 
\begin{align*}
\lascoux_{(0,2)}^{(\beta)} \cdot \lascoux_{(0,1)}^{(\beta)} &= \lascouxatom_{(0,3)}^{(\beta)} + \lascouxatom_{(1,2)}^{(\beta)} + 2\beta\lascouxatom_{(1,3)}^{(\beta)} + \lascouxatom_{(2,1)}^{(\beta)} + \beta\lascouxatom_{(2,2)}^{(\beta)} \\
 &+  \beta^2\lascouxatom_{(2,3)}^{(\beta)} + \lascouxatom_{(3,0)}^{(\beta)} + 2\beta\lascouxatom_{(3,1)}^{(\beta)} + \beta^2\lascouxatom_{(3,2)}^{(\beta)}. 
\end{align*} 
We have checked Conjecture~\ref{conj:multiplying_lascouxs} by computer for all $a,b$ such that $|a| \leq 5, |b| \leq 5,$ and $a$ and $b$ have at most three zeros.  Specializing Conjecture~\ref{conj:multiplying_lascouxs} at $\beta=0$ recovers a well-known conjecture of V.~Reiner and M.~Shimozono on products of key polynomials (see \cite{Pun} for discussion and partial results).

\subsection{Proof of Theorem~\ref{thm:lDecomp}}
First, we need a straightforward operation on weak compositions. 
Following \cite{Assaf.Searles:2}, given a weak composition $a$, we define a {\bf left swap} to be the exchange of two entries $a_i \le a_j$ where $i<j$. 

\begin{definition}[{\cite[\textsection 3.2]{Assaf.Searles:2}}]\label{def:lswap}
Given a weak composition $a$, let $\lswap(a)$ be the set of weak compositions $b$ that can be obtained from $a$ by a (possibly empty) sequence of left swaps.
\end{definition}

\begin{example}\label{ex:lswap}
For the weak compositions $a = (0,1,2)$ and $b = (0,3,1)$, we have
\[
\lswap(a) = \{(0,1,2), (1,0,2), (1,2,0), (0,2,1), (2,0,1), (2,1,0)\}
\]
and
\[ \pushQED{\qed} \lswap(b) = \{(0,3,1), (3,0,1), (1,3,0), (3,1,0)\}. \qedhere \popQED\] \let\qed\relax
\end{example}

The following characterization appears as \cite[Lemma 3.1]{Searles}.
\begin{lemma}\label{lem:dominic's_lemma}
 For any weak composition $a$, 
 \[ \pushQED{\qed}
 \lswap(a) = \{ b : \sort(b) = \sort(a) \text{ and } w(b) \le w(a) \}. \qedhere \popQED
 \]  \let\qed\relax
\end{lemma}

Hence, to prove Theorem \ref{thm:lDecomp}, it suffices by Lemma~\ref{lem:dominic's_lemma} to construct a weight-preserving bijection 
\begin{equation}\label{eq:psibar_map}
 \overline{\psi} :  \LPSSF(a) \rightarrow \coprod_{b \in \lswap(a)} \LASSF(b). 
\end{equation}

We begin by constructing a column-set preserving (and hence weight-preserving) bijection  
\begin{equation}\label{eq:psi_map}
\psi :  \KSSF(a) \rightarrow \coprod_{b \in \lswap(a)} \ASSF(b)
\end{equation}
for the non-set-valued case.
The advantage of Equation~\eqref{eq:psi_map} over Equation~\eqref{eq:psibar_map} is that we know {\it a priori} that there must exist a weight-preserving bijection between $\KSSF(a)$ and $\coprod_{b \in \lswap(a)} \ASSF(b)$, since both sets are known to generate the Demazure character $\key_a$ \cite{HLMvW11:LRrule}, \cite{Mason:atom}. To our knowledge, an explicit bijection between these sets has, however, not appeared previously in the literature. Our goal is to give an explicit bijection $\psi$ that moreover preserves the column sets of the fillings. Having established $\psi$ with this property, we will find it straightforward to extend $\psi$ to the desired map $\overline{\psi}$ of Equation~\eqref{eq:psibar_map}, thereby proving the theorem.

Let $T\in \KSSF(a)$. Define $\psi(T)$ as follows. Let $i_1$ be the least entry in the first column of $T$. For $k > 1$, recursively define $i_k$ to be the greatest entry in column $k$ that is weakly less than $i_{k-1}$, terminating when there is no such entry. Place the entries $i_1, i_2, \dots$ to form the bottom row of $\psi(T)$ in row index $i_1$, while deleting them from $T$. Repeat this process on what remains of $T$ to find the next-lowest row of $\psi(T)$, etc. In \cite{Searles}, this algorithm is referred to as \emph{left row-filling}; by \cite[Lemma 5.2, Lemma 5.3]{Searles}, $\psi(T) \in \ASSF(b)$ for some $b$. (Strictly speaking, the domain of the map defined in \cite{Searles} consists of reverse semistandard Young tableaux rather than the fillings of $\KSSF(a)$. However, as the map clearly operates at the level of column sets (the positions of boxes in a column are irrelevant) and as the column sets of $T \in \KSSF(a)$ can clearly be reordered to create a reverse semistandard Young tableau, this distinction is insignificant.) 

\begin{example}\label{ex:psi}
Suppose that $a = (1,2,0,3,3)$.  Then we have, for example, 
\begin{center}
\begin{displaymath}
 \KSSF(1,2,0,3,3) \ni \begin{ytableau} 
*(gray) 1 & \textbf{1} \\
*(gray) 2 & \textbf{2} & \textbf{1}\\
*(gray) 3 \\ 
*(gray) 4 & \textbf{3} & \textbf{3} & \textbf{2}  \\
*(gray) 5 &  \textbf{5} & \textbf{4} & \textbf{3} \\
\end{ytableau}   \xRightarrow{\hspace*{.25in}\psi\hspace*{.25in}} 
 \begin{ytableau}
 \none & \textbf{5} &\textbf{4} & \textbf{2} \\
\none \vline \\
\none & \textbf{3} & \textbf{3} & \textbf{3} \\
\none & \textbf{2} \\
\none & \textbf{1}   & \textbf{1} \\
\end{ytableau} \in \ASSF(2,1,3,0,3).
\end{displaymath}
\end{center}
Observe that $b =(2,1,3,0,3) \in \lswap(a)$, as desired.
\end{example}

It is clear that $b$ always satisfies $\sort(a) = \sort(b)$.
We need to establish the stronger property that $b \in \lswap(a)$. To this end, we make use of the fact (observed in \cite{Assaf.Schilling}) that the fillings in $\KSSF(a)$ are exactly the same objects (flipped upside-down) as the \emph{semistandard key tableaux} of \cite{Assaf:nonsymmetric}, and hence are in straightforward bijection with \emph{Kohnert diagrams}. 

A {\bf Kohnert move} \cite{Kohnert} on a finite set of boxes in $\mathbb{N}\times \mathbb{N}$ (realized as the lattice points of the first quadrant of the plane) moves the rightmost box in some row down the the highest empty space below it in the same column. The {\bf Kohnert diagrams} of $D(a)$ are the box diagrams that can be obtained from $D(a)$ by a (possibly empty) sequence of Kohnert moves. A semistandard key tableau records a particular choice of Kohnert moves to achieve a given Kohnert diagram; the number in each box of a semistandard key tableau represents the row that box has moved to in the corresponding Kohnert diagram. 

\begin{example}\label{ex:keytabtoKohnert}
Figure~\ref{fig:keytabtoKohnert} shows the semistandard key tableau obtained by flipping the filling $T \in \KSSF(a)$ of Example~\ref{ex:psi} upside-down, together with the corresponding Kohnert diagram. We shade the boxes of the Kohnert diagram in \textcolor{red}{red} to distinguish them from empty space.
\end{example}

\begin{figure}[h]
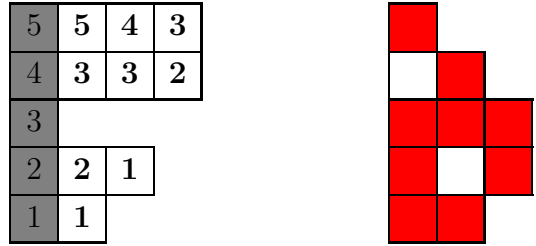

\begin{center}
\begin{displaymath}
\begin{ytableau} 
*(gray) 5 &  \textbf{5} & \textbf{4} & \textbf{3} \\
*(gray) 4 & \textbf{3} & \textbf{3} & \textbf{2}  \\
*(gray) 3 \\ 
*(gray) 2 & \textbf{2} & \textbf{1}\\
*(gray) 1 & \textbf{1} \\
\end{ytableau}
\qquad \qquad \qquad
 \begin{ytableau}
 *(red) \blank \\
 \blank  & *(red) \blank \\
 *(red) \blank  & *(red) \blank & *(red) \blank \\
*(red) \blank  & \blank &  *(red)\blank  \\
*(red)  \blank  & *(red) \blank   \\ 
\end{ytableau}
\end{displaymath}
\caption{\label{fig:keytabtoKohnert} The semistandard key tableau equivalent to the filling $T \in \KSSF(a)$ of Example~\ref{ex:psi}, together with its corresponding Kohnert diagram.}
\end{center}
\end{figure}

This reinterpretation in terms of Kohnert diagrams is useful because it facilitates a diagrammatic understanding of the $\lswap$ operation.

\begin{lemma}\label{lem:lswapKohnert}
Let $a$ and $b$ be weak compositions. Then $b\in \lswap(a)$ if and only if $D(b)$ is a Kohnert diagram of $D(a)$.
\end{lemma} 
\begin{proof}
Suppose $b\in \lswap(a)$. Then $b$ is obtained from $a$ by a sequence of swaps that move larger entries leftwards. Any such swap can be achieved by Kohnert moves: given $D(a)$, to swap a row of length $r$ with a lower row of length $r'$, where $0\le r'\le r$, move the rightmost $r-r'$ boxes of the higher row down to the lower row, in order from right to left. Clearly, these moves are all valid Kohnert moves and so $D(b)$ can be realized as a Kohnert diagram of $D(a)$.

Conversely, suppose $D(b)$ is a Kohnert diagram of $D(a)$. Since Kohnert moves keep each box in its column, we have $\sort(b) = \sort(a)$. It remains to observe that the only way one can rearrange rows of a skyline diagram via Kohnert moves is to move the ``overhang'' of a longer row down to join onto a (possibly empty) shorter row. But this operation corresponds to an application of the $\lswap$ operation, so $b \in \lswap(a)$.
\end{proof}

\begin{definition}\label{def:nearestskyline}
Suppose that $a$ is a weak composition and $K$ is a Kohnert diagram of $D(a)$. Define the {\bf nearest skyline} diagram of $K$ to be the skyline diagram $K^\uparrow$ obtained as follows: Reading rows of $K$ from left to right, top to bottom, find the first box $\mathfrak{b}$ of $K$ with an empty space immediately to its left; move $\mathfrak{b}$ upwards to the first available space above it. Repeat this operation until all boxes are in left-justified rows, i.e., until a skyline diagram is obtained. 
\end{definition}

An example of the computation of a nearest skyline diagram appears in Figure~\ref{fig:nearest_skyline}.

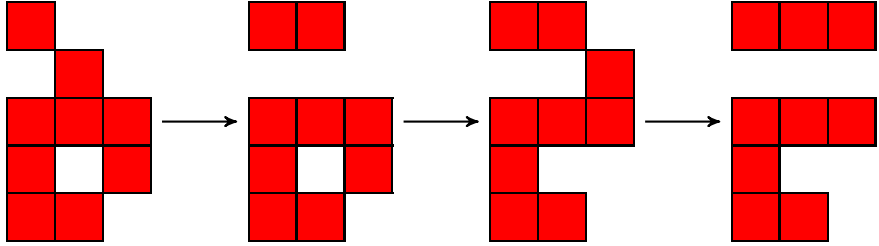
\begin{figure}[h]
\begin{tikzpicture}
\node (A) { \begin{ytableau}
 *(red) \blank \\
 \none  & *(red) \blank \\
 *(red) \blank  & *(red) \blank & *(red) \blank \\
*(red) \blank  & \blank &  *(red)\blank  \\
*(red)  \blank  & *(red) \blank   \\ 
\end{ytableau}};
\node[right = 1 of A] (B) { \begin{ytableau}
 *(red) \blank & *(red) \blank \\
 \none   \\
 *(red) \blank  & *(red) \blank & *(red) \blank \\
*(red) \blank  & \blank &  *(red)\blank  \\
*(red)  \blank  & *(red) \blank   \\ 
\end{ytableau}};
\node[right = 1 of B] (C) { \begin{ytableau}
 *(red) \blank & *(red) \blank \\
 \none & \none &  *(red)\blank  \\
 *(red) \blank  & *(red) \blank & *(red) \blank \\
*(red) \blank  & \none   \\
*(red)  \blank  & *(red) \blank   \\ 
\end{ytableau}};
\node[right = 1 of C] (D) { \begin{ytableau}
 *(red) \blank & *(red) \blank   &  *(red)\blank \\
 \none & \none \\
 *(red) \blank  & *(red) \blank & *(red) \blank \\
*(red) \blank  & \none   \\
*(red)  \blank  & *(red) \blank   \\ 
\end{ytableau}};
\path (A) edge[posex]   (B);
\path (B) edge[posex]   (C);
\path (C) edge[posex]   (D);
\end{tikzpicture}

\caption{The iterative computation of the nearest skyline diagram of the Kohnert diagram from Figure~\ref{fig:keytabtoKohnert}.}\label{fig:nearest_skyline}
\end{figure}

\begin{lemma}\label{lem:nearestisKohnert}
If $a$ is a weak composition and $K$ is a Kohnert diagram of $D(a)$, then $K^\uparrow$ is also a Kohnert diagram of $D(a)$.
\end{lemma}
\begin{proof}
Since $K$ is a Kohnert diagram of $D(a)$, $D(a)$ can be obtained by a sequence of reverse Kohnert moves on $K$. Let $\mathfrak{b}$ be the box moved by an iteration of the ``nearest skyline'' algorithm of Definition~\ref{def:nearestskyline}. Since all boxes strictly higher than $\mathfrak{b}$ are by definition in left-justified rows, $\mathfrak{b}$ does not land left of another box in the same row. So Definition~\ref{def:nearestskyline} uses valid reverse Kohnert moves. We will show these reverse Kohnert moves on $K$ (or other sequences of reverse Kohnert moves resulting in the same diagram) are all necessary to obtain \emph{any} skyline diagram of which $K$ is a Kohnert diagram, hence are necessary to obtain $D(a)$. 

In order for a skyline diagram to be obtained, since a reverse Kohnert move can never cause a box to land the left of another box in the same row, any box of $K$ that has an empty space to its left in its row must at some stage move upwards. Moreover, a reverse Kohnert move on any box does not cause the left-justification status of other boxes in the diagram to change: those boxes that were in left-justified rows still are, and those boxes that had an empty space to their left still do. 

Now observe that when the algorithm raises a box $\mathfrak{b}$, it moves boxes within the column of $\mathfrak{b}$ by the minimal amount possible in order to raise $\mathfrak{b}$. In the case $\mathfrak{b}$ jumps over other boxes, it is of course possible the same diagram may be obtained by moving the boxes above $\mathfrak{b}$ first and then moving $\mathfrak{b}$ upwards by a smaller distance. But this gives rise to the same diagram: the entire interval of boxes weakly above $\mathfrak{b}$ moves up one space. Since Kohnert diagrams do not distinguish between individual boxes, this move is thus the minimal move on a Kohnert diagram that raises $\mathfrak{b}$. 

So the algorithm moves only the boxes that have to be moved upwards to obtain a skyline diagram, it moves boxes upwards by the minimal distance needed to achieve this, and the set of boxes of $K$ that need to be moved upwards is independent of the order in which these boxes are moved. So all these reverse Kohnert moves (or equivalent sequences of reverse Kohnert moves resulting in the same diagrams) are necessary in order to obtain $D(a)$, and thus $K^\uparrow$ is a Kohnert diagram of $D(a)$.
\end{proof}

We can reinterpret $\psi$ to act on Kohnert diagrams. Let $K$ be a Kohnert diagram of $D(a)$. We first decompose $K$ into {\bf threads} $\theta_1, \theta_2, \dots$. Let $\mathfrak{b}^1$ be the lowest box in column $1$ of $K$. Then for $k>1$, recursively define $\mathfrak{b}^k$ to be the highest box in column $k$ of $K$ that is weakly lower than $\mathfrak{b}^{k-1}$. Let $\theta_1 = \{ \mathfrak{b}^1, \mathfrak{b}^2, \dots \}$. After deleting $\theta_1$, repeat this process on the remainder of $K$ to obtain $\theta_2$, etc. Now, define $\psi(K)$ to be the skyline filling given by placing the row indices of the boxes in thread $\theta_k$ into row $k$ in weakly decreasing order.

\begin{example}\label{ex:threading}
The left of Figure~\ref{fig:threading} shows the Kohnert diagram $K$ of Figure~\ref{fig:keytabtoKohnert} with threads indicated both by labels in the boxes and also by coloring. On the right of Figure~\ref{fig:threading} is the corresponding filling $\psi(K) \in \ASSF(2,1,3,0,3)$. 
Observe that the filling $\psi(K)$ of Figure~\ref{fig:threading} coincides with the filling obtained in Example~\ref{ex:psi}.
\end{example}

\begin{figure}[h]
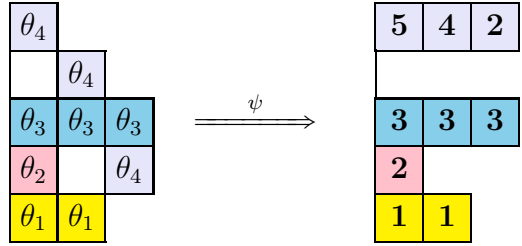

\begin{center}
\begin{displaymath}
 \begin{ytableau}
*(Lavender)  \theta_4 \\
 \blank  & *(Lavender) \theta_4 \\
  *(SkyBlue) \theta_3  &  *(SkyBlue) \theta_3 & *(SkyBlue) \theta_3 \\
 *(Pink) \theta_2 & \blank &  *(Lavender) \theta_4  \\
 *(yellow) \theta_1  & *(yellow)  \theta_1   \\ 
\end{ytableau} \quad \xRightarrow{\hspace*{.25in}\psi\hspace*{.25in}} 
 \begin{ytableau}
 \none & *(Lavender) \textbf{5} & *(Lavender) \textbf{4} & *(Lavender) \textbf{2} \\
\none \vline \\
\none & *(SkyBlue) \textbf{3} & *(SkyBlue) \textbf{3} & *(SkyBlue) \textbf{3} \\
\none & *(Pink) \textbf{2} \\
\none & *(yellow) \textbf{1}   & *(yellow) \textbf{1} \\
\end{ytableau}
\end{displaymath}
\end{center}
\caption{\label{fig:threading} The threading of the Kohnert diagram $K$ of Figure~\ref{fig:keytabtoKohnert} and the corresponding filling $\psi(K) \in \ASSF(2,1,3,0,3)$. Here, we have labeled the boxes of $K$ by the threads of which they are part. The color-coding of the boxes of $K$ is redundant with this labeling by threads, and the color-coding of the boxes of $\psi(K)$ matches the color-coding of $K$.}
\end{figure}

The following lemma justifies the mild abuse of notation in using the same symbol $\psi$ both for a map on skyline fillings and for a map on Kohnert diagrams.

\begin{lemma}\label{lem:psis_agree}
Let $T \in \KSSF(a)$ and let $K$ be the corresponding Kohnert diagram of $D(a)$. Then, $\psi(T) = \psi(K)$.
\end{lemma}
\begin{proof}
Observe that the definition of $\psi(T)$ only depends on the sets of labels appearing in each column of $T$. This set of labels coincides with the set of row indices of the boxes in the corresponding column of $K$. It is then straightforward from unwinding the definitions that $\psi(T) = \psi(K)$, as desired.
\end{proof}

\begin{lemma}\label{lem:threading}
The threading of a Kohnert diagram $K$ is preserved under the ``nearest skyline'' algorithm of Definition~\ref{def:nearestskyline}; that is, whenever the algorithm moves a box upwards, the set of boxes in each thread and the location of the leftmost box in each thread both remain unchanged.
\end{lemma}
\begin{proof}
Any box in the first column of $K$ has no empty space to its left and thus is never moved by the  ``nearest skyline'' algorithm. Hence, the location of the leftmost box in each thread does not change during construction of $K^\uparrow$.

To see that threads retain the same set of boxes, it is enough by induction to consider a single application of the ``box raising'' operation. Suppose this operation takes $K$ to $K^+$ by acting on the box $\mathfrak{b}^c$ in column $c$. Suppose $\mathfrak{b}^c$ is in the thread $\theta_i = \{ \mathfrak{b}^1, \mathfrak{b}^2, \dots \}$ in $K$, where $\mathfrak{b}^k$ is in column $k$. Certainly, $c > 1$, so $\mathfrak{b}^{c-1}$ exists. 

First, consider the case that $\mathfrak{b}^c$ moves up one space, i.e., $K$ has no box immediately above $\mathfrak{b}^c$ in its column. By definition, $K$ has no box immediately to the left of $\mathfrak{b}^c$ in its row, so $\mathfrak{b}^{c-1}$ is strictly above $\mathfrak{b}^c$ in $K$. Thus, $\mathfrak{b}^c$ still lies in the same thread as $\mathfrak{b}^{c-1}$ in $K^+$. By definition, $\mathfrak{b}^c$ is not left of any box in its row in $K^+$, so $K^+$ has $\mathfrak{b}^c$ in the same thread as $\mathfrak{b}^{c+1}$, if such a box exists. Thus, in this case the threading of $K$ coincides with the threading of $K^+$.

Now, suppose $\mathfrak{b}^c$ ``jumps'' over at least one box, i.e., $K$ has a box immediately above $\mathfrak{b}^c$ in its column. Since $K$ has all boxes above $\mathfrak{b}^c$ in left-justified rows, $\mathfrak{b}^{c-1}$ is necessarily the rightmost box of the lowest row above $\mathfrak{b}^c$ that ends in column $c-1$. Since the rows that $\mathfrak{b}^c$ jumps over have length at least $c$, $\mathfrak{b}^c$ remains threaded with $\mathfrak{b}^{c-1}$ in $K^+$. By definition, $\mathfrak{b}^c$ is not left of any box in its row in $K^+$. In the rows that $\mathfrak{b}^c$ jumps over, observe that both $K$ and $K^+$ have each box threaded with all the boxes to its left in its row. In particular, $K^+$ does not thread $\mathfrak{b}^c$ with with a box in column $c+1$ from a row that has been jumped over. So $\mathfrak{b}^c$ remains threaded with $\mathfrak{b}^{c+1}$ in $K^+$, if such a box exists. Thus, the threadings of $K$ and $K^+$ coincide.
\end{proof}

\begin{lemma}\label{lem:shapes_the_same}
Let $K$ be a Kohnert diagram. Then, the shape of $\psi(K)$ is the nearest skyline diagram $K^\uparrow$.
\end{lemma}
\begin{proof}
By Lemma~\ref{lem:threading}, the threading of $K$ is the same as the threading of $K^\uparrow$. Hence, the shape of $\psi(K)$ equals the shape of $\psi(K^\uparrow)$. Since $K^\uparrow$ is a skyline diagram, it is clear from the definitions that the threads of $K^\uparrow$ are just its rows. Hence, the shape of $\psi(K^\uparrow)$ is the skyline diagram $K^\uparrow$ itself. Thus, the shape of $\psi(K)$ is $K^\uparrow$.
\end{proof}

\begin{lemma}\label{lem:column_sets_distinct}
Let $T, U \in \KSSF(a)$ be distinct. Then, there is a column in which $T$ and $U$ do not have the same set of labels.
\end{lemma}
\begin{proof}
Let $K_T$ and $K_U$ be the Kohnert diagrams corresponding to $T$ and $U$, respectively. Then $T$ and $U$ have the same labels in column $c$ exactly if $K_T$ and $K_U$ have boxes in the same positions in column $c$.  Since $K_T \neq K_U$, there is some column $\hat{c}$ in which their boxes are not in identical positions. Hence, the set of labels in column $\hat{c}$ of $T$ differs from the set of labels in column $\hat{c}$ of $U$.
\end{proof}

\begin{theorem}
The map
\[ \psi :  \KSSF(a) \rightarrow \coprod_{b \in \lswap(a)} \ASSF(b)\]
is well-defined, and is a column-set-preserving (and thus weight-preserving) bijection.
\end{theorem}
\begin{proof}
Let $T\in \KSSF(a)$. We know from \cite{Searles} that $\psi(T) \in \ASSF(b)$ for some $b$. Let $K$ be the Kohnert diagram corresponding to $T$ and suppose the nearest skyline diagram $K^\uparrow$ has shape $d$. By definition, $K$ is a Kohnert diagram of $D(a)$. Hence by Lemma~\ref{lem:nearestisKohnert}, we have that $D(d)$ is a Kohnert diagram of $D(a)$.  Thus by Lemma~\ref{lem:lswapKohnert}, $d \in \lswap(a)$. By Lemma~\ref{lem:shapes_the_same}, $\psi(K)$ has shape $d$. Thus, by Lemma~\ref{lem:psis_agree}, $\psi(T)$ has shape $d$, and so  $\psi(T)\in \ASSF(d) \subseteq \coprod_{b \in \lswap(a)} \ASSF(b)$.

By definition, $\psi$ preserves column sets. Since no two elements of $\KSSF(a)$ have identical column sets by Lemma~\ref{lem:column_sets_distinct}, this implies $\psi$ is injective. Since $\KSSF(a)$ and $\coprod_{b \in \lswap(a)} \ASSF(b)$ both generate the Demazure character $\key_a$  \cite{HLMvW11:LRrule}, \cite{Mason:atom}, they are equinumerous (and finite). Hence $\psi$ is a bijection.
\end{proof}

\begin{lemma}
The bijection $\psi$ extends to a column-set-preserving (and thus weight-preserving) bijection
\[ \overline{\psi} :  \LPSSF(a) \rightarrow \coprod_{b \in \lswap(a)} \LASSF(b), \]
as follows:
For $\overline{T} \in \LPSSF(a)$, let $T$ be the filling obtained by deleting the free entries of $\overline{T}$. Then, $\overline{\psi}(\overline{T})$ is given by placing each free entry in $\overline{T}$ with the smallest possible anchor in the corresponding column of $\psi(T)$, subject to the decreasingness condition (S.2).
\end{lemma}
\begin{proof}
Let $T \in \KSSF(a)$ and let $S = \psi(T) \in \ASSF(b)$.
The only thing that needs to be checked is that, for any valid assignment of free entries to columns of $T$, there is a corresponding valid assignment of the same free entries to the same columns of $S$, and vice versa.

Suppose for a contradiction that this is false for some $\overline{T} \in \LPSSF(a)$, whose anchor entries form the filling $T$. Let $k$ be the greatest free entry of column $c$ of $\overline{T}$ that cannot be added as a free entry in column $c$ of $S$ without violating the decreasingness condition (S.2).  Then, for every entry $m$ greater than $k$ in column $c$ of $S$, the entry immediately right in column $c+1$ of $S$ is strictly greater than $k$.  
However, since $\psi$ preserves column sets, the entries of each column of $S$ are a permutation of the entries of the corresponding column of $T$. Thus, $T$ also has at least as many entries that are strictly greater than $k$ in column $c+1$  as it has entries that are that are strictly greater than $k$ in column $c$.  
Hence by (S.2) for $T$, every entry in column $c+1$ of $T$ that is strictly greater than $k$ must be immediately right of an entry that is strictly greater than $k$ in column $c$. Thus, $k$ cannot be added as a free entry in column $c$ of $T$, contradicting the existence of $\overline{T}$. The argument for the other direction is identical.
\end{proof}

This completes the proof of Theorem~\ref{thm:lDecomp}. \qed

\subsection{The quasiLascoux expansion of a Lascoux polynomial}
The Lascoux polynomials expand positively in the basis of quasiLascoux polynomials. 
Following \cite{Assaf.Searles:2}, define $\Qlswap(a)$ to be all $b\in \lswap(a)$ such that if $c\in \lswap(a)$ and $b^+ = c^+$, then $c \ge b$. 

\begin{example}[cf.~Example~\ref{ex:lswap}]
For the weak compositions $a=(0,1,2)$ and $b = (0,3,1)$, we have
\[
 \Qlswap(a) = \{(0,1,2), (0,2,1)\}
\]
and
\[ \pushQED{\qed}  \Qlswap(b) = \{ (0,3,1), (1,3,0) \} . \qedhere \popQED \] \let\qed\relax 
\end{example}

\begin{theorem}\label{LtoQL}
For any weak composition $a$, we have
\[\lascoux_a^{(\beta)} = \sum_{b\in \Qlswap(a)} \qlascoux_b^{(\beta)}.\]
In particular, every Lascoux polynomial $\lascoux_a^{(\beta)}$ is a positive sum of quasiLascoux polynomials.
\end{theorem}
\begin{proof}
Suppose that $b\in \Qlswap(a)$ and that $c$ is another weak composition with $c\ge b$ and $c^+ = b^+$. Then clearly $c\in \lswap(a)$. By the definitions of $\lswap$ and $\Qlswap$, every $c\in \lswap(a)$ is either in $\Qlswap(a)$ or else dominates some $b\in \Qlswap(a)$ with $c^+ = b^+$. This establishes the second equality in the following:
\[\lascoux_a^{(\beta)} \,\,\,\,\,\, = \!\! \sum_{c\in \lswap(a)} \lascouxatom_c^{(\beta)} \,\,\,\,\,\, = \!\! \sum_{b\in \Qlswap(a)} \! \sum_{\substack{c\ge b \\ c^+ = b^+}} \!\! \lascouxatom_c^{(\beta)} \,\,\,\,\,\, = \!\! \sum_{b\in \Qlswap(a)} \qlascoux_b^{(\beta)},\]
where the first equality is by combining Theorem~\ref{thm:lDecomp} and Lemma~\ref{lem:dominic's_lemma}, and the third equality is by Definition~\ref{def:QL}.
\end{proof}

\begin{remark}
The $\beta=0$ specialization of Theorem~\ref{LtoQL} recovers the expansion of Demazure characters in the quasikey basis \cite[Theorem~3.7]{Assaf.Searles:2}.
\end{remark}

\section*{Acknowledgements}
OP was partially supported by a Mathematical Sciences Postdoctoral Research Fellowship (\#1703696) from the National Science Foundation. OP is grateful for educational conversations with Allen Knutson, Emily Sergel, and Ardea Thurston-Shaine.  CM was partially supported by a GAANN Fellowship from the Department of Mathematics, University of Illinois at Urbana-Champaign.

%%%%%%%%%%%%%%%%%%%%%%%%%%%%%%%%%%%%%%%%%%%%%%%%%%%%%%%%%%%%
%
%  Bibliography
%
%%%%%%%%%%%%%%%%%%%%%%%%%%%%%%%%%%%%%%%%%%%%%%%%%%%%%%%%%%%%

\bibliographystyle{amsalpha} 
\bibliography{quasilascoux}

\end{document}